\documentclass[a4paper,12pt]{article}
 \usepackage[english]{babel}
  \usepackage{amsmath,amsfonts,amssymb,amsthm}
  \usepackage{paralist}

  \usepackage{hyperref} 

 \usepackage[active]{srcltx} 
 \usepackage{color,soul} 

\newtheorem{theorem}{Theorem}[section]
\newtheorem{corollary}[theorem]{Corollary}
\newtheorem{lemma}[theorem]{Lemma}
\newtheorem{proposition}[theorem]{Proposition}

\theoremstyle{definition}
\newtheorem{definition}[theorem]{Definition}
\newtheorem{remark}{Remark}

\newtheorem{counter}{Counter-example}
\newtheorem{open}{Open problem}
\newtheorem{conjecture}{Conjecture}

\def\N{\mathbb{N}}

\def\R{\mathbb{R}}

\let\e=\varepsilon
\let\vp=\varphi
\let\t=\tilde
\let\ol=\overline
\let\ul=\underline
\let\.=\cdot
\let\0=\emptyset
\let\mc=\mathcal
\def\ex{\exists\;}
\def\fa{\forall\;}

\def\pp{,\dots,}

\def\O{\Omega}

\def\loc{\text{\rm loc}}

\def\dist{\text{\rm dist}}

\def\supp{\text{\rm supp}}
\def\tr{\text{\rm Tr}}

\def\inter{\text{\rm int\,}}
\def\norma#1{\|#1\|_\infty}

\newcommand{\su}[2]{\genfrac{}{}{0pt}{}{#1}{#2}}

\def\eq#1{{\rm(\ref{eq:#1})}}
\def\thm#1{Theorem \ref{thm:#1}}
\def\seq#1{(#1_n)_{n\in\N}}
\def\limn{\lim_{n\to\infty}}

\def\LL{\mathcal{L}}

\def\pe{principal eigenvalue}
\def\MP{maximum principle}
\def\SMP{strong maximum principle}
\def\l{\lambda_1}

\newenvironment{formula}[1]{\begin{equation}\label{eq:#1}}
                       {\end{equation}\noindent}

\def\Fi#1{\begin{formula}{#1}}
\def\Ff{\end{formula}\noindent}

\setstcolor{red}

\def\uguale{\stackrel{u_0}{=}}
\def\minore{\stackrel{u_0}{\leq}}

\title{
\bf{Generalizations and properties of the principal eigenvalue of elliptic
operators in unbounded domains}}
\author{Henri Berestycki\thanks{
CAMS - \'Ecole des Hautes \'Etudes en Sciences Sociales,
190-198 avenue de France, 75013, Paris, France, {\itshape e-mail:
}\ttfamily hb@ehess.fr}\ \
and Luca Rossi\thanks{ Dipartimento di Matematica,
  Universit\`a di Padova,
via Trieste 63, 35121 Padova, Italy, {\itshape e-mail:
}\ttfamily lucar@math.unipd.it}}

\date{}

\begin{document}

\maketitle

\begin{abstract}
Using three different notions of generalized \pe\ of linear second order
elliptic operators in unbounded domains, we derive necessary and
sufficient conditions for the validity of the \MP, as well as for
the existence of positive eigenfunctions for the Dirichlet
problem. Relations between these \pe s, their
simplicity and several other properties are further discussed.
\end{abstract}

\medskip

{\em MSC 2010: }{Primary} 35J15; {Secondary} 35B09, 35B50, 35P15.

{\em Keywords: }Elliptic equations, unbounded domains, principal eigenvalue,
positive solutions, maximum  principle.


\tableofcontents


\section{Definitions and main results}


\subsection{Introduction}
What is the principal eigenvalue of a general linear second order
elliptic operator in an unbounded
domain associated with Dirichlet conditions~?
Under what conditions do such operators satisfy the \MP~? When do positive
eigenfunctions exist~? These are some of the themes
we discuss in this paper.

The Krein-Rutman theory provides the existence of the principal
(or first) eigenvalue $\lambda_\O$
of an elliptic operator $-L$ in a bounded smooth domain $\O$, under Dirichlet
boundary condition. This eigenvalue is the bottom of
the spectrum of $-L$, for the Dirichlet problem, it is simple
and the associated eigenfunction is positive in $\O$.
The positivity of $\lambda_\O$ guarantees the existence of a unique solution to
the inhomogeneous Dirichlet problem. These properties,
together with several others, have been extended by
H.~Berestycki, L.~Nirenberg and S.~R.~S.~Varadhan \cite{BNV} to the
case of {\em bounded non-smooth} domains by introducing the
notion of the {\em generalized \pe}.

In the present paper, we consider the
case of {\em unbounded domains}, continuing the
study begun in \cite{BHRossi}, in collaboration with F.~Hamel, and in
\cite{BR1}. Our aim is to emphasize the implications of
unboundedness of the domain rather than lack of smoothness.
For this reason, some of our results are stated for domains with smooth
boundaries even though the techniques of \cite{BNV} would allow one to extend
them to non-smooth domains. Let us also say from the outset that rather than
adopting a functional analytical point
of view, the object of study here are the very partial differential equations
(or inequalities) associated with the eigenvalue problem. We consider general
linear second order not necessarily self-adjoint operators. 
As we show here, some of the basic properties of the
\pe\ fail in general in the unbounded case.

In \cite{BHRossi}, it has been pointed out that the generalized \pe\ $\l$
of \cite{BNV} is not suited for characterizing the existence
of solutions for a class of semilinear problems in unbounded domains. It
is further shown that another quantity - denoted by $\l'$ - provides the right
characterization.

Here, we introduce still another quantity, $\l''$, which turns out
to provide a sufficient condition for the validity
of the maximum principle in unbounded domains. The main
object of this paper is to investigate the
relations between the three quantities $\l$, $\l'$, $\l''$ and their properties.
The relations
between $\l$ and $\l'$ have been established in \cite{BHRossi}, \cite{BR1},
but only in low dimension.
Here, we improve them to arbitrary dimension.

The forthcoming paper \cite{BNR}, in collaboration with G.~Nadin, deals with
extensions to parabolic operators of the notions introduced here. 
We also examine there the relationship of these notions with Lyapunov exponent
type ideas. 
Applications to nonlinear problems will be further discussed in \cite{BNR}.


\subsection{Motivations: semilinear problems, maximum principle and
eigenfunctions}
\label{sec:motivation}

Our interest for the generalization of the notion of \pe\ to unbounded domains
originally stemmed from the study of the Fisher-KPP reaction-diffusion equation
$$
\partial_t u-a_{ij}(x)\partial_{ij}u-b_i(x)\partial_i u=f(x,u),\quad
t>0,\ x\in\R^N,
$$
which arises for instance in some models in population dynamics.
In such models, the large time behavior of the population
- and in particular its persistence or extinction - is determined by
the existence
of a unique positive stationary solution. This, in turn,
depends on the sign of the \pe\ associated with the linearized operator
about $u\equiv0$:
$$
\LL u=a_{ij}(x)\partial_{ij}u+b_i(x)\partial_i u+f_s(x,0)u.
$$
When the coefficients of the equation do not depend on $x$,
the right notion of \pe\ is the quantity $\l$ introduced in \cite{BNV}, whereas
when the coefficients are periodic and $\LL$ is self-adjoint, it is the periodic
\pe\ (see \cite{BHR1}). 
In the general case considered in \cite{BHRossi}, one needs to consider both
$\l$ and $\l'$, the latter being a kind of generalization of
the periodic \pe.

Furthermore, the study of asymptotic spreading speeds for
general Fisher-KPP equations like the one above
involves principal eigenvalues of families of associated linear operators.
Building on the results and related notions to the ones presented here,
properties about the
asymptotic spreading speed for general non-homogeneous equations are
established in \cite{BN}.

Another motivation for our study comes from a very basic
question: 
does the sign of the generalized \pe\ $\l$ characterize the validity of the \MP\
for bounded solutions to linear
equations in unbounded domains ? This is known to be the case for bounded
domains. We
show here that the answer is no.
The necessary and sufficient conditions for the \MP\ will be shown here to
hinge on $\l'$ and on another generalization of the \pe, denoted by $\l''$.

It is also a very natural question in itself to determine what are the
eigenvalues associated with positive eigenfunctions for the Dirichlet condition,
as well as their multiplicities, for general operators in unbounded  domains.


\subsection{Hypotheses and definitions}\label{sec:hyp}

Throughout the paper, $\O$ denotes a domain in $\R^N$ (in general
unbounded and possibly non-smooth) and $L$ a general
elliptic operator in non-divergence form:
$$Lu=a_{ij}(x)\partial_{ij} u + b_i(x)\partial_i u +c(x)u$$
(the usual convention for summation from $1$ to $N$ on repeated
indices is adopted). When we say that $\O$ is smooth we mean that it is
of class
$C^{1,1}$. We use the notation $\ul\alpha(x),\
\ol\alpha(x)$ to indicate respectively the smallest and the
largest eigenvalues of the symmetric matrix $(a_{ij}(x))$,
i.~e.
$$\ul\alpha(x):=\min_{\su{\xi\in\R^N}{|\xi|=1}}a_{ij}(x)\xi_i\xi_j,\qquad
\ol\alpha(x):=\max_{\su{\xi\in\R^N}{|\xi|=1}}a_{ij}(x)\xi_i\xi_j.$$
The basic assumptions on the coefficients of $L$ are:
$$a_{ij}\in C^0(\ol\O),\qquad \fa x\in\ol\O,\quad\ul\alpha(x)>0,\qquad
b_i,c\in L^\infty_\loc(\ol\O).$$
These hypotheses will always be understood, unless otherwise specified, since
they are needed in most of our results.
Note that we allow the ellipticity of $(a_{ij})$ to
degenerate at infinity. Also, $C^0(\ol\O)$ denotes the space of
functions which are continuous on $\ol\O$, but not necessarily
bounded.
Additional hypotheses will be explicitly required in some of
the statements below.
The operator $L$ is said to be uniformly elliptic if $\inf_\O\ul\alpha>0$
and is termed self-adjoint if it can be written in the form
$$Lu=\partial_i(a_{ij}(x)\partial_j u) + c(x)u.$$

It is well known that if the domain $\O$ is bounded and smooth then
the
Krein-Rutman theory (see \cite{KR}) implies the existence of a unique real
number $\lambda=\lambda_\O$
such that the problem
$$
\left\{\begin{array}{ll}
-L\vp=\lambda\vp & \text{a.e.~in }\O\\
\vp=0 & \text{on }\partial\O\\
\end{array}\right.
$$
admits a positive solution $\vp\in W^{2,p}(\O)$, $\fa p<\infty$.
The quantity $\lambda_\O$ and
the associated eigenfunction $\vp$ (which is unique up to a
multiplicative constant) are respectively called Dirichlet \pe\
and eigenfunction of $-L$ in $\O$. Henceforth, we keep the notation
$\lambda_\O$ for this Dirichlet \pe.

The Krein-Rutman theory cannot be applied if $\O$ is non-smooth
or unbounded (except for problems in periodic settings), because the resolvent
of $-L$ is not compact. However, the fundamental properties of the
Dirichlet \pe\ have been extended in \cite{BNV} to the case of non-smooth
bounded domains considering the following notion:
\Fi{l1}\l(-L,\O):=\sup\{\lambda\ :\ \ex\phi\in
W^{2,N}_\loc(\O),\ \phi>0,\ (L+\lambda)\phi\leq0\text{
a.e.~in }\O\}.\Ff
If $\O$ is bounded and smooth,
then $\l(-L,\O)$ coincides with the classical Dirichlet \pe\ $\lambda_\O$.
An equivalent definition was previously given by S.~Agmon in \cite{A1}
in the case of operators in divergence form defined on Riemannian
manifolds and, for general operators, by R.~D. Nussbaum and Y.~Pinchover 
\cite{NP}, building on a result by M.~H. Protter and H.~F. Weinberger 
\cite{Max2}.

The quantity defined by \eq{l1} is our first notion of a generalized \pe\ in an
unbounded
domain. We also consider here two other generalizations.

\begin{definition}\label{def:l1}
For given $\O$ and $L$, we set
\begin{equation}\label{eq:l1'}\begin{split}
\l'(-L,\O):=\inf\{\lambda\ :\ \ex\phi\in W^{2,N}_\loc(\O)\cap
L^\infty(\O),\ \phi>0,\ (L+\lambda)\phi\geq0\text{ a.e.~in
}\O,\\
\fa\xi\in\partial\O,\ \lim_{x\to\xi}\phi(x)=0\};
\end{split}
\end{equation}
\Fi{l1''}\l''(-L,\O):=\sup\{\lambda\ :\ \ex\phi\in
W^{2,N}_\loc(\O),\ \inf_\O\phi>0,\ (L+\lambda)\phi\leq0\text{
a.e.~in }\O\}.\Ff
\end{definition}

The quantity $\l'$ has been introduced in \cite{BHR1},
\cite{BHRossi} and it also coincides with $\lambda_\O$
if $\O$ is bounded and smooth.
However, in contradistinction with $\l$, it is equal to the periodic \pe\ when
$\O$ and $L$ are periodic. Later on we will show that these two
properties are fulfilled by $\l''$ as well.

If $\O$ is smooth then the three quantities
$\l(-L,\O)$, $\l'(-L,\O)$, $\l''(-L,\O)$ -if finite- are eigenvalues for $-L$ 
in $\O$ under Dirichlet boundary
conditions. This follows from Theorems \ref{thm:spettro} and \ref{thm:relations}
part (ii) below and the obvious inequality $\l''(-L,\O)\leq\l(-L,\O)$.
But, as shown in Section \ref{sec:simplicity},
 the \pe s $\l'$, $\l''$ do not have in general admissible eigenfunctions,
i.e.~eigenfunctions satisfying the additional requirements of being bounded
from above or having positive infimum far from $\partial\O$ respectively.

It may occur that the sets in the definitions \eq{l1}, \eq{l1'} or \eq{l1''}
are empty (see Section \ref{sec:finite}). In such cases, we set
$\l(-L,\O):=-\infty$, $\l'(-L,\O):=+\infty$,
$\l''(-L,\O):=-\infty$ respectively. A sufficient (yet not necessary)
condition for the sets in \eq{l1}, \eq{l1''} to be nonempty is:
$\sup_\O c<\infty$, as is immediately seen by taking $\phi\equiv1$
in the formulas. We will find that $\l'<+\infty$ when $\O$ is
smooth as a consequence of a comparison result between $\l$ and $\l'$,
\thm{relations} part (ii).
If $\O$ is non-smooth then the boundary condition in \eq{l1'}
is too strong a requirement, and one should relax it in the sense
of \cite{BNV}. However, we do not stress the non-smooth aspect in the present
paper.
In Section \ref{sec:other}, we show that, if $\O$ is uniformly smooth
and $L$ is uniformly elliptic and has bounded coefficients, then
the definition \eq{l1''} of $\l''$ does not change if the condition $\inf\phi>0$
is only required in any subset of $\O$
having positive distance from $\partial\O$. This condition is
more natural because it is
satisfied by the classical Dirichlet principal eigenfunction when $\O$ is
bounded and smooth.

The admissible functions for $\l'$ and $\l''$ are bounded respectively
from above and from below by (a positive constant times) the function
$\beta\equiv1$.
Considering instead an arbitrary barrier $\beta$ yields further extensions of
these definitions.

\begin{definition}\label{def:l1b}
For given $\O,\ L$ and positive function $\beta:\O\to\R$, we set
\[\begin{split}
\lambda'_\beta(-L,\O):=\inf\{\lambda\ :\ \ex\phi\in W^{2,N}_\loc(\O),\
0<\phi\leq\beta,\ (L+\lambda)\phi\geq0\text{ a.e.~in
}\O,\\
\fa\xi\in\partial\O,\ \lim_{x\to\xi}\phi(x)=0\};
\end{split}\]
\[\lambda''_\beta(-L,\O):=\sup\{\lambda\ :\ \ex\phi\in
W^{2,N}_\loc(\O),\ \phi\geq\beta,\ (L+\lambda)\phi\leq0\text{
a.e.~in }\O\}.\]
\end{definition}

If $\l',\ \l''$ arise
in the study of the existence and uniqueness of positive {\em bounded} solutions
for
the Dirichlet problem, $\lambda'_\beta,\ \lambda''_\beta$ come into play when
considering solutions with prescribed maximal (or minimal) growth $\beta$.
We will mainly focus here on $\l'$ and $\l''$, but we also derive properties
with $\lambda'_\beta,\ \lambda''_\beta$ along the way.


\subsection{Statement of the main results}

We start with investigating the existence of eigenvalues
associated with positive eigenfunctions satisfying Dirichlet boundary
conditions. These are given by the problem
\Fi{Dev}
\left\{\begin{array}{ll}
-L\vp=\lambda\vp & \text{a.e.~in }\O\\
\vp=0 & \text{on $\partial\O\ $ (if $\O\neq\R^N$)}
\end{array}\right.
\Ff

\begin{definition}\label{def:ev}
We say that $\lambda\in\R$ is an {\em eigenvalue} of $-L$ in $\O$ (associated
with positive eigenfunction), under Dirichlet boundary condition,
if the problem \eq{Dev} admits a positive solution $\vp\in W^{2,p}_\loc(\ol\O)$,
$\fa p<\infty$.
Such a solution is called (positive) {\em eigenfunction} and the set of all
eigenvalues is denoted by $\mc{E}$.
\end{definition}

In the following, since we only deal with positive eigenfunctions, we omit
to mention it.
If $\O$ is bounded and smooth then it is well known that $\mc{E}=\{\lambda_\O\}$
and that this eigenvalue is simple.
This property is improved in \cite{BNV} to non-smooth domains, by replacing
$\lambda_\O$ with $\l(-L,\O)$ and imposing
the boundary conditions on a suitable subset of $\partial\O$.
The picture changes drastically in the case of unbounded domains.
Indeed, in Section \ref{sec:charE} below, we derive the following
characterization.

\begin{theorem}\label{thm:spettro}
If $\O$ is unbounded and smooth then
$\mc{E}=(-\infty,\l(-L,\O)]$.
\end{theorem}

%
%

%

\thm{spettro} improves the property already known that the set of
eigenvalues associated with eigenfunctions {\em without
prescribed conditions} on $\partial\O$ coincides with
$(-\infty,\l(-L,\O)]$ (see, e.g., \cite{A1}). Note that
in the case of bounded smooth domains this property still holds if
one prescribes the Dirichlet condition on a proper subset of the
boundary. 
This example might lead one to believe that the reason why $\mc{E}$ does
not reduce to a singleton when $\O$ is unbounded
is that no Dirichlet condition is imposed at infinity.
Counter-example \ref{c-e} in Section \ref{sec:charE}
shows that this is not the case, even if one imposes an
exponential decay.

Recently, we came across the work \cite{Furusho1} by
Y.~Furusho and Y.~Ogura (1981) that does not seem to be very well
known. In that paper, they prove \thm{spettro} but only in the
case where $\O$ is an exterior smooth domain (and $L$ has smooth
coefficients). The main difficulty when dealing with general
unbounded domain is that, in order to construct a solution, one needs to control
the behavior near the boundary of a family of solutions in bounded
domains. We
achieve this
by use of an appropriate version of the {\em boundary Harnack
inequality} (also known as Carleson estimate) 
due to H.~Berestycki, L.~Caffarelli and L.~Nirenberg \cite{BCN3}. 

Let us further point out that, if $L$ has H\"older continuous coefficients,
the problem of the existence of eigenfunctions
(vanishing on $\partial\O$) can also be approached by using the
Green function and the Martin boundary theory (see, e.g., \cite{Pinch94}).
However, as far as we know, the result
of \thm{spettro} was not previously derived in the generality in which we state
it here.
\\

Next, we derive a necessary and a sufficient condition, expressed
in terms of $\l'$ and of $\l''$ respectively, for the validity of the \MP\ in
unbounded domains. With \MP\ we mean the following:

\begin{definition}\label{def:MP}
We say that the operator $L$ satisfies the {\em\MP\ }
({\em MP }for short) in $\O$ if every
function $u\in W^{2,N}_\loc(\O)$ such that
$$Lu\geq0 \ \text{ a.e.~in }\O,\qquad
\sup_{\O}u<\infty,\qquad
\fa\xi\in\partial\O,\ \ \limsup_{x\to\xi}u(x)\leq0,$$
satisfies $u\leq0$ in $\O$.
\end{definition}

Note that no conditions are imposed at infinity, except for the
boundedness from above. This condition is redundant if $\O$ is
bounded. In the case of bounded smooth domains, it is well known
that the MP holds iff $\lambda_\O>0$. This result is
improved in \cite{BNV} to bounded non-smooth domains by replacing
$\lambda_\O$ with $\l(-L,\O)$ and considering a refined version of
the \MP. The extensions and results in the general theory of \cite{BNV} are
recalled in the Appendix \ref{sec:BNV} here. For
unbounded domains, we will show that the validity of the MP is not
related to the sign of $\l$ (even if one restricts Definition
\ref{def:MP} to subsolutions decaying exponentially to $0$, see
Counter-example \ref{c-e}), but rather to those of $\l'$ and $\l''$.

\begin{theorem}\label{thm:MP}
The operator $L$ satisfies the MP in $\O$
\begin{enumerate}[(i)]
\item if $\;\l''(-L,\O)>0$ and the coefficients of $L$ satisfy
\Fi{ABC3}
\sup_\O c<\infty,
\qquad
\limsup_{\su{x\in\O}{|x|\to\infty}}\frac{|a_{ij}(x)|}{|x|^2}<\infty,
\qquad
\limsup_{\su{x\in\O}{|x|\to\infty}}\frac{b(x)\.x}{|x|^2}<\infty;
\Ff

\item only if $\;\l'(-L,\O)\geq0$.
\end{enumerate}
\end{theorem}

Condition \eq{ABC3} \marginpar{*} is specific to the metric of $\R^N$. 
Actually, many of the results here can be extended to more general Remaniann 
manifolds. For this purpose, this condition (1.5) should be modified by 
involving the corresponding metric.

Y.~Pinchover pointed out to us that the hypothesis on the $a_{ij}$ in \eq{ABC3}
is sharp for statement (i) to hold. Indeed, it can be proved that the operator
$Lu=(1+|x|)^{2+\e}\Delta u-u$ in $\R^N$, with $\e>0$ and
$N\geq3$, satisfies $\l''(-L,\R^N)\geq1$
but the equation $Lu=0$ in $\R^N$ admits positive bounded solutions
(actually, one can show that $\l'(-L,\R^N)=-\infty$). Moreover, it is easy
to construct operators with $b_i(x)=O(|x|^{1+\e})$ for which $\l''>0$ and the MP
does not hold.
Let us mention that the continuity of
$(a_{ij})$ is not used in the proof of \thm{MP}, and that the
ellipticity is only required to hold locally uniformly in $\O$.
\thm{MP} is a particular case of \thm{bMP} below, which asserts that the
MP holds for subsolutions satisfying $\sup u/\beta<\infty$
instead of $\sup u<\infty$ - for a given barrier function $\beta$
growing at most exponentially - if $\lambda_\beta''(-L,\O)>0$ and only if
$\lambda_\beta'(-L,\O)\geq0$, where $\lambda_\beta'$, $\lambda_\beta''$ are
given by Definition \ref{def:l1b}. In the case of operators with bounded
coefficients and subsolutions bounded from above, the implication
$\l''>0\Rightarrow\,$MP is implicitly contained in Lemma 2.1 of \cite{BHRossi}.
Note that if $\O$ is bounded (possibly
non-smooth) then Proposition 6.1 of \cite{BNV} yields $\l''(-L,\O)=\l(-L,\O)$.
This is why, in that case, $\l(-L,\O)>0\Rightarrow\,$MP.
In the limiting case where $\l'$ and $\l''$ are equal to $0$, the MP might
or might not hold
(see Remark \ref{rem:MP0} below).\\

Next, we derive some relations between the generalized \pe s $\l,\ \l'$
and $\l''$.
\begin{theorem}\label{thm:relations}
Let $\O$ be smooth. Then, the following properties hold:

\begin{enumerate}[(i)]
\item if $L$ is self-adjoint and the $a_{ij}$ are bounded then
  $\lambda_1(-L,\O)=\lambda_1'(-L,\O)$;

\item for general $L$ it holds that
$\l'(-L,\O)\leq\l(-L,\O)$;

\item under the growth condition \eq{ABC3} it holds that
$\l''(-L,\O)\leq\l'(-L,\O)$.
\end{enumerate}
\end{theorem}

From the above result and the definitions of $\l$ and $\l'$ it follows that,
if a self-adjoint operator $L$ with $a_{ij}\in L^\infty(\O)$ admits a {\em 
bounded}
(positive) eigenfunction associated with an eigenvalue $\lambda\in\R$,  
then necessarily
$\lambda=\l(-L,\O)=\l'(-L,\O)$.

We actually prove \thm{relations} part (i) with $\lambda_1'(-L,\O)$ replaced by
$\lambda_\beta'(-L,\O)$, for any barrier
$\beta$ with subexponential growth.

In the case of uniformly elliptic operators with bounded
smooth coefficients, the inequality $\l'\leq\l$ was proved in
\cite{BR1} in dimension $1$, together
with the inequality $\l'\geq\l$ for self-adjoint operators
in dimension less than $4$ (subsequently improved to non-smooth
operators in \cite{Pinch07}).
The question in arbitrary dimension was stated as an open problem. With the
results here, it is now
completely solved. Instead, the relations between $\l'$ and
$\l''$ are not fully understood; \thm{relations} part (iii) gives only
a partial information. Indeed, we do not know any example of operators for which
$\l''<\l'$.
We leave it as an open problem to prove the following

\begin{conjecture}\label{conj}
If $\O$ is smooth and $L$ has bounded coefficients then
$\l'(-L,\O)=\l''(-L,\O)$.
\end{conjecture}

We are able to prove Conjecture \ref{conj} in some particular cases, where we
actually show that all three notions of generalized \pe s coincide.

\begin{theorem}\label{thm:SC}
Let $\O$ be unbounded and smooth. Then
$\l(-L,\O)=\l''(-L,\O)$ ($=\l'(-L,\O)$ if \eq{ABC3} holds)
in each of the following cases:
\begin{enumerate}[1)]
\item $L$ is a self-adjoint, uniformly elliptic operator with bounded 
coefficients and either
$N=1$ or $\O=\R^N$ and $L$ is radially symmetric;

\item $L=\t L+\gamma(x)$, where $\t L$ is an elliptic operator
such that $\l(-\t L,\O)=\l''(-\t L,\O)$ and $\gamma\in L^\infty(\O)$ is
nonnegative and satisfies
$\displaystyle\lim_{\su{x\in\O}{|x|\to\infty}}\gamma(x)=0$;

\item
$$\l(-L,\O)\leq-\limsup_{\su{x\in\O}{|x|\to\infty}}c(x);$$

\item the $a_{ij}$ are bounded,
$L$ is uniformly elliptic and it is either self-adjoint or in
non-divergence form with $\displaystyle\lim_{\su{x\in\O}{|x|\to\infty}}b(x)=0$,
and
$$\fa r>0,\
\fa\beta<\limsup_{\su{x\in\O}{|x|\to\infty}}c(x),\quad
\ex B_r(x_0)\subset\O\ \text{ s.~t.~}\ \inf_{B_r(x_0)} c>\beta.$$
\end{enumerate}
\end{theorem}

We remark
that the hypothesis on $c$ in the case 4 of \thm{SC} is fulfilled
if $\O=\R^N$ and $c(x)\to \gamma(x/|x|)$ as $|x|\to\infty$, with
$\gamma$ lower semicontinuous.
Cases 2-4 will be derived from a general result
- \thm{equivalence} below - which provides a useful characterization for
$\l''$. One of the tools used in its proof is an extension of the
boundary Harnack inequality to inhomogeneous Dirichlet problems. 
Another tool is a continuity property of $\l(-L,\O)$ with respect to
perturbations of the domain $\O$. In particular, we derive the following
continuity property with respect to {\em exterior} perturbations, which is of
independent interest.

\begin{theorem}\label{thm:l1dec}
Let $(\O_n)_{n\in\N}$ be a family of domains
such that $\O_1\backslash\O$ is bounded, $\partial\O$ is smooth in a
neighborhood of $\ol{\O_1\backslash\O}$ and
$$\fa n\in\N,\quad\O_n\supset\O_{n+1}\supset\O,
\qquad\bigcap_{n\in\N}\ol\O_n=\ol\O.$$
Then $\lambda_1(-L,\O_n)\nearrow\lambda_1(-L,\O)$ as
$n\to\infty$.
\end{theorem}

In the above statement, it is understood that the coefficients of $L$
satisfy the hypotheses of Section \ref{sec:hyp} in $\O_1$ and not only in $\O$.
\thm{l1dec} is an important feature of the principal eigenvalue $\l$.
Contrary
to interior convergence of domains (cf.~Proposition \ref{l1} part (iv) below),
continuity with respect to exterior perturbations is a
subtle issue and it may possibly fail (see, e.g., \cite{Arrieta}, \cite{Daners}
for the case of bounded domains). We discuss several aspects of this
property in Section \ref{sec:domain}.

In Section \ref{sec:simplicity}, we discuss the existence of admissible
eigenfunctions for $\l'$, $\l''$, as  well as the simplicity of $\l$.
A sufficient condition for the latter is derived by
using the notion of solution of {\em minimal growth at infinity}. This is in
the spirit but a slightly different version of the notion 
introduced by S.~Agmon in his pioneering and important paper \cite{A1}.
Combining this condition with Theorems \ref{thm:spettro}, \ref{thm:MP}
and the characterization of $\l''$ given in \thm{equivalence}, we are able to
extend
the basic properties of the classical Dirichlet \pe\ to
the case of unbounded domains, provided that $c$ is negative at
infinity.

\begin{proposition}\label{pro:c<0}
Let $\O$ be unbounded and smooth and let
$$\xi:=\limsup_{\su{x\in\O}{|x|\to\infty}}c(x).$$ The following properties hold:

\begin{enumerate}[(i)]
\item if $\xi<0$ and \eq{ABC3} holds then $L$ satisfies the MP in $\O$
iff $\l(-L,\O)>0$;

\item if $\l(-L,\O)<-\xi$ then any positive function $v\in
W^{2,N}_{loc}(\O)$ satisfying $(L+\l(-L,\O))v\leq0$ a.e.~in $\O$ coincides, up
to a scalar multiple, with the eigenfunction $\vp_1$ associated with
$\l(-L,\O)$. Moreover, $\vp_1$ is bounded and, if the coefficients of $L$ are
bounded, it decays exponentially to $0$.

\end{enumerate}
\end{proposition}

Actually, hypothesis \eq{ABC3} is not required in the ``only if'' implication of
statement (i). If $\xi=0$  then $\l(-L,\O)>0$ does not imply the MP,
even if $c<0$ everywhere (see Remark \ref{rem:c<0}).
Statement (ii) was announced
and used in our previous paper \cite{BR3}.
There, we dealt with Neumann problems in
smooth infinite cylinders. This led us to define a notion of generalized \pe\
which incorporates the Neumann boundary condition. However, the proof presented
below works
exactly at the same way in that case.
The last statement of Proposition \ref{pro:c<0} part (ii) follows from a general
result about the exponential decay of subsolutions of the
Dirichlet problem - Proposition \ref{pro:decay} below.
\\

We conclude by investigating the continuity
of $\l$ with respect to the coefficients, as well as its behavior as
the size of the zero and the second order coefficients blows up or the
ellipticity degenerates.
\\

Let us point out that some of the results concerning
$\l$ and $\l''$ still hold if
$\O$ is not connected. This is seen by
noticing that $\l(-L,\O)$ and $\l''(-L,\O)$ coincide with the infimum
of the $\l$ and $\l''$ in the connected components of $\O$. Exceptions
are: the results about the existence of eigenfunctions, such as \thm{spettro},
the implication MP$\,\Rightarrow\l>0$ in Proposition \ref{pro:c<0} part (i)
(unless $\O$ has
a finite number of connected components). 
Note that $\l'$ is equal to the supremum of the $\l'$ in the
connected components of $\O$. We further remark that, if $\O$ is connected, the
definition \eq{l1} of $\l$ does not change if one replaces $\phi>0$ with
$\phi\geq0$, $\phi\not\equiv0$. This is no longer true if $\O$ is not connected.

As was already mentioned above, most of the results of this paper can be 
extended to the case of linear
elliptic equations on noncompact manifolds. There are only few points, such as
condition \eq{ABC3},
where the volume growth of balls and other properties of $\R^N$ 
are used and need to be adapted to this more general setting.


\section{Preliminary considerations on the definitions and assumptions}


\subsection{Exploring other possible definitions}\label{sec:other}

To start with, we address the question of what happens if one enlarges the
class of admissible functions in definition \eq{l1''}.
For $\e>0$, we set
$$\O^\e:=\{x\in\O\ :\ \dist(x,\partial\O)>\e\}.$$

\begin{proposition}\label{pro:l1''=}
Let $\O$ be uniformly of class $C^{2,1}$ and $L$ be a uniformly
elliptic operator with $a_{ij}$, $b_i$ bounded and $c$ bounded from above.
Then, the quantity
$\l''(-L,\O)$ defined by \eq{l1''} satisfies
$$\l''(-L,\O)=\sup\{\lambda\ :\ \ex\phi\in
W^{2,N}_\loc(\O),\fa\e>0,\ \inf_{\O^\e}\phi>0,\ (L+\lambda)\phi\leq0\text{
a.e.~in }\O\}.$$
\end{proposition}

\begin{proof}
To prove the statement it is sufficient to show that if $\lambda\in\R,\
\phi\in W^{2,N}_\loc(\O)$ satisfy
$$\fa\e>0,\quad \inf_{\O^\e}\phi>0,\qquad (L+\lambda)\phi\leq0\text{
a.e.~in }\O,$$ then every $\t\lambda<\lambda$ belongs to the set in
\eq{l1''}. We can assume without loss of generality that
$\lambda=0$, so that $\t\lambda<0$, and that $\O\neq\R^N$. For $x\in\ol\O$, set
$d(x):=\dist(x,\O)$. Since $\O$ is uniformly of class $C^{2,1}$, we know from
\cite{LN} that, for $\e>0$ small enough, the distance function $d$ belongs to
$W^{2,\infty}(\O\backslash\O^\e)$. Furthermore, $|\nabla d|=1$ in 
$\O\backslash\O^\e$.
Define the function $v(x):=\cos(kd(x))$, where $k$ is a positive constant that
will be chosen later. For a.e.~$x\in\O\backslash\O^\e$ it holds that
\[\begin{split}
Lv &=-ka_{ij}(x)[kv\partial_i d\partial_j
d+\sin(kd)\partial_{ij}d]
-k\sin(kd)b_i(x)\partial_i d+c(x)v\\
&\leq(-k^2\ul\alpha(x)+c(x))v+(Ck+k|b(x)|)|\sin(kd)|,
\end{split}\]
where $C=\sum_{i,j}\norma{a_{ij}\partial_{ij}d}$.
Hence, since $v\geq|\sin(kd)|$
in $\O\backslash\O^{\frac\pi{4k}}$, setting $\delta:=\min(\e,\frac\pi{4k})$ we
get
$$Lv\leq(-k^2\inf_\O\ul\alpha+\sup_\O c+Ck+k\sup_\O|b|)v\quad\text{a.e.~in
}\O\backslash\O^\delta.$$
It is then possible to choose $k>0$ in such a way that $Lv\leq0$
a.e.~in $\O\backslash\O^\delta$. Let $\chi:\R\to[0,+\infty)$ be a
smooth cutoff function satisfying
$$\chi=1\quad\text{in }[0,1/2],\qquad
\chi=0\quad\text{in }[1,+\infty).$$
Then, for $x\in\O$, define $w(x):=v(x)\chi(\frac1\delta d(x))$.
The function $w$ is nonnegative, smooth, belongs to $W^{2,\infty}(\O)$, vanishes
on $\ol\O^\delta$
and it satisfies
$$\inf_{\O\backslash\O^{\delta/2}}w>0,\qquad
Lw\leq0\quad\text{a.e.~in }\O\backslash\O^{\delta/2}.$$ We
finally set $\t\phi(x):=h\phi(x)+w(x)$, for some positive constant
$h$. This function satisfies
$$\inf_\O\t\phi\geq\min\left(h\inf_{\O^{\delta/2}}\phi,\inf_{\O\backslash\O^{
\delta/2}}w\right)>0,\qquad
L\t\phi\leq0\quad\text{a.e.~in }\O\backslash\O^{\delta/2}.$$
Moreover, for a.e.~$x\in\O^{\delta/2}$,
$$(L+\t\lambda)\t\phi\leq h\t\lambda\phi+(L+\t\lambda)w.$$
Therefore, for $h$ large enough, we have that
$(L+\t\lambda)\t\phi\leq0$ a.e.~in $\O$. This shows that
$\t\lambda$ belongs to the set in \eq{l1''}.
\end{proof}

The above proof leads us to formulate the following. 

\begin{open}
Does the result of Proposition \ref{pro:l1''=} hold true if one drops the
uniform ellipticity
and boundedness of the coefficients of $L$ ?
\end{open}

Starting from definition \eq{l1}, one could define several quantities
by replacing ``$\sup$'' with ``$\inf$'', $(L+\lambda)\phi\leq0$ with
$(L+\lambda)\phi\geq0$ as well as by adding the conditions
$\sup\phi<\infty$ or $\inf\phi>0$ (or, more generally,
$\sup\frac\phi\beta<\infty$ or $\inf\frac\phi\beta>0$ for a given barrier
function $\beta$).
Let us explain why we focus on the ones in Definition \ref{def:l1} and their
extensions with a barrier function $\beta$.

First of all, it is clear that if $c\in L^\infty(\O)$ then
replacing $\sup$ with
$\inf$ in definition \eq{l1} gives $-\infty$, whereas taking
$(L+\lambda)\phi\geq0$ instead of $(L+\lambda)\phi\leq0$ gives $+\infty$.
This is true even if one adds the conditions
$\sup\phi<\infty$ or $\inf\phi>0$.

Two other possibilities are thus left.
$$\t{\lambda}_1'(-L,\O):=\sup\{\lambda\ :\ \ex\phi\in
W^{2,N}_\loc(\O)\cap L^\infty(\O),\ \phi>0,\
(L+\lambda)\phi\leq0\text{ a.e.~in }\O\},$$
$$\t{\lambda}_1''(-L,\O):=\inf\{\lambda\ :\ \ex\phi\in
W^{2,N}_\loc(\O),\ \inf_\O\phi>0,\ (L+\lambda)\phi\geq0\text{
a.e.~in }\O\}.$$

One can show that, if $L$ is a uniformly elliptic operator with bounded
coefficients, then $\t{\lambda}_1''(-L,\O)=-\infty$.
Instead, if $c$ is bounded from above, the quantity $\t{\lambda}_1'(-L,\O)$
is a well defined real number satisfying
$-\sup_\O c\leq\t{\lambda}_1'(-L,\O)\leq\l(-L,\O)$.
However, its sign is not related to the validity of the MP.
This is seen, by means of \thm{MP}, considering the operator $L$ defined
in Counter-example \ref{c-e}, that satisfies
$$\l''(-L,\R)\leq\l'(-L,\R)<0<\l(-L,\R)=\t{\lambda}_1'(-L,\R).$$
%
%



\subsection{Previously known properties of $\l$ and $\l'$}\label{sec:known}

In this section, we present some known properties of $\l$ and
$\l'$. We recall that, for a bounded smooth domain $\O$, $\lambda_\O$
denotes the classical \pe\ of $-L$ in $\O$ under Dirichlet boundary conditions.

\begin{proposition}\label{l1}
The generalized \pe\ $\l(-L,\O)$ defined by \eq{l1} satisfies the following
properties:
\begin{enumerate}[(i)]

\item
if $\O$ is bounded and smooth then $\l(-L,\O)=\lambda_\O$;

\item
$$-\sup_\O c\leq\lambda_1(-L,\O)\leq Cr^{-2},$$
where $0<r\leq1$
is the radius of some ball $B$ contained in $\O$ and $C>0$ only depends
on $N$, $\inf_B\ul\alpha$ and the $L^\infty(B)$ norms of $a_{ij}$,
$b_i$, $c$;

\item
if $\O'\subset\O$ then
$\lambda_1(-L,\O')\geq\lambda_1(-L,\O)$, with strict inequality if
$\O'$ is bounded and $|\O\backslash\O'|>0$;

\item
if $(\O_n)_{n\in\N}$ is a family of nonempty domains
such that
$$\O_n\subset\O_{n+1},\qquad\bigcup_{n\in\N}\O_n=\O,$$
then $\lambda_1(-L,\O_n)\searrow\lambda_1(-L,\O)$ as
$n\to\infty$;

\item
if $\l(-L,\O)>-\infty$ then
there exists a positive function $\varphi\in
W^{2,p}_\loc(\O)$, $\fa p<\infty$, satisfying
\Fi{pef}
-L\vp=\l(-L,\O)\vp\quad\text{a.e.~in }\O;
\Ff

\item
if $L$ is self-adjoint then \Fi{RR}
\lambda_1(-L,\O)=\inf_{\su{\phi\in C^1_c(\O)}{\phi\not\equiv0}}
\frac{\int_\O(a_{ij}(x)\partial_i\phi\partial_j\phi-c(x)\phi^2)}
{\int_\O\phi^2},\Ff
where $C^1_c(\O)$ denotes the space of compactly supported, $C^1$ 
functions in $\O$. 
In particular, $\l(-L,\O)$ is nondecreasing with
respect to the matrix $(a_{ij})$;

\item
in its dependence on $c$, $\l(-L,\O)$ is
nonincreasing (i.~e.~$h\geq0$ in $\O$ implies $\l(-(L+h),\O)\leq\l(-L,\O)$),
concave and Lipschitz-continuous (using the $L^\infty$ norm) with
Lipschitz constant 1;

\item
for uniformly elliptic operators with bounded coefficients,
$\l(-L,\O)$ is locally Lipschitz-continuous with respect to the
$b_i$, with Lipschitz constant depending only on $N,\ \O$, the
ellipticity constants and the $L^\infty$ norm of $c$.

\end{enumerate}
\end{proposition}

The above properties, in particular (i), motivate the terming of ``generalized
\pe''. Property (i) can be
deduced from a mini-max formula in \cite{Max}. The
upper bound
in (ii) is Lemma 1.1 of \cite{BNV}.
The lower bound follows immediately from the definition,
as does the inequality $\geq$ in (iii).
The strict inequality in the case of bounded $\O'$ is given by
Theorem 2.4 of \cite{BNV} (and it actually holds in greater generality,
cf.~Remark \ref{rem:strict} below).
The proofs of (iv), (v) can be found in \cite{A1} in the
case of operators with smooth coefficients (see also \cite{BNV} for general
operators in bounded, non-smooth domains), but the same arguments apply to the
general case.
We point out that if $\O$ is smooth then \thm{spettro} above
is a much stronger result than Proposition \ref{l1} part (v),
providing in particular a function $\vp$ satisfying in addition
the Dirichlet boundary condition.
Property (vi) follows from (i), (iv) and the Rayleigh-Ritz
variational formula for the classical Dirichlet \pe, as shown in
\cite{A1}.
Properties (vii) and (viii) are respectively Propositions 2.1 and 5.1 in
\cite{BNV}.

%

\begin{remark}\label{rem:l1monot}
The monotonicity of $\l$ with respect to $c$ is strict if $\O$
is bounded, whereas it might not be the case if $\O$ is unbounded (see
the proof of Proposition \ref{pro:nopef} below for an example).
Likewise, the decreasing monotonicity with respect to the domain given by
Proposition
\ref{l1} part (iii)
might not be strict in the case of unbounded domains. For example, in dimension
1, the operator $Lu=u''$ clearly satisfies $\l(-L,\R)=\l(-L,\R_+)=0$.
\end{remark}

The generalized \pe\ $\l'(-L,\O)$
also coincides with the Dirichlet
\pe\ $\lambda_\O$ if $\O$ is bounded and smooth.
Moreover, it coincides with the periodic \pe\
$\lambda_p$ under Dirichlet boundary conditions if
$\O$ is smooth and $\O$ and $L$ are periodic with the same period. We
say that $\O$ is periodic, with period $(l_1\pp l_N)\in\R_+^N$,
if $\O+\{l_i e_i\}=\O$ for $i=1\pp N$, where $\{e_1\pp e_N\}$ is the
canonical basis of $\R^N$; the operator $L$ is said to be periodic, with period
$(l_1\pp l_N)$, if its coefficients are periodic with the same period $(l_1\pp
l_N)$. We recall that
$\lambda_p$ is the unique real number $\lambda$ such that the problem
\eq{Dev} admits a positive periodic solution. Such a solution, which is
unique up to a multiplicative constant, is called periodic principal
eigenfunction.

\begin{proposition}\label{l1'}
The generalized \pe\ $\l'(-L,\O)$ defined by \eq{l1'} satisfies the following
properties:
\begin{enumerate}[(i)]
\item
if $\O$ is bounded and smooth then
$\l'(-L,\O)=\lambda_\O$;

\item
if $\O$ is smooth then $\l'(-L,\O)<+\infty$ and,
if in addition \eq{ABC3} holds, then $\l'(-L,\O)\in\R$;

\item
if $\O$ is smooth and $\O,\ L$ are periodic, with the same period, then
$\lambda'_1(-L,\O)=\lambda_p$.
\end{enumerate}
\end{proposition}

The fact that the set of ``admissible functions'' in \eq{l1'} could be empty was
not discussed in previous
papers, where, essentially, only the case $\O=\R^N$ and $L$ with
bounded coefficients was treated.
Statements (i) and (iii) are proved in \cite{BHRossi} (for
operators with smooth coefficients) and \cite{PhD}, requiring, in both cases,
the additional condition that the functions $\phi$ in \eq{l1'}
are uniformly Lipschitz-continuous. In the general case where this extra
condition is not imposed, properties (i), (ii) and the
inequality $\l'\geq\lambda_p$ in (iii) can be deduced from the properties of
$\l$ and $\l''$ - Propositions \ref{l1}, \ref{pro:l1''} - by means of
\thm{relations} here. The inequality $\l'\leq\lambda_p$ is
immediately obtained by taking $\phi$ equal to the periodic principal
eigenfunction in \eq{l1'}.

%
%


\subsection{Finiteness of $\l$}\label{sec:finite}

By Proposition \ref{l1} part (ii), we know that $\l\in\R$ if $c$ is bounded
above. Otherwise, it could be equal to $-\infty$, that is, the set of
admissible functions in \eq{l1}
could be empty.

\begin{proposition}
Let $\O$ be a smooth domain and $L$ be a uniformly elliptic
operator with $a_{ij},\ b_i$ bounded and $c$ such that there
exists a positive constant $\delta$ and a sequence $\seq{x}$
satisfying
\Fi{supc}
\fa n\in\N,\quad B_\delta(x_n)\subset\O,\qquad
\limn\inf_{B_\delta(x_n)}c=+\infty.
\Ff
Then, $\l(-L,\O)=-\infty$
and the MP, as stated in Definition \ref{def:MP}, does not hold
for $L$ in $\O$.
\end{proposition}

\begin{proof}
Since, for $\lambda\in\R$, $L-\lambda$ satisfies the same condition \eq{supc}
as $L$ and, by definition,
$$\l(-L,\O)=\l(-(L-\lambda),\O)-\lambda,$$
to prove that $\l(-L,\O)=-\infty$ it is sufficient to show that
$\l(-L,\O)\leq0$. Thus, owing to Proposition \ref{l1} part (iii), it is
enough to show that
$\l(-L,B_\delta(x_n))\leq0$ for some $n\in\N$. 
Consider a function $\vartheta\in C^2([0,\delta])$ satisfying
$$\vartheta>0\; \text{ in }[0,\delta),\qquad
\vartheta'(0)=0,\qquad\vartheta(\delta)=\vartheta'(\delta)=0,\qquad
\vartheta''>0\; \text{ in }[\frac\delta2,\delta]$$
(for instance, $\vartheta(r):=\cos(\frac\pi\delta r)+1$).
The functions $\seq{\theta}$ defined by $\theta_n(x):=\vartheta(|x-x_n|)$
satisfy, a.e.~in $B_\delta(x_n)\backslash B_{\frac\delta2}(x_n)$,
$$a_{ij}(x)\partial_{ij}\theta_n+b_i(x)\partial_i\theta_n\geq
\left(\inf_{\O}\ul\alpha\right)\vartheta''(|x-x_n|)-k
|\vartheta'(|x-x_n|)+\vartheta(|x-x_n|)|,$$
for some $k$ independent of $n$.
There exists then $\rho\in(0,\delta)$, independent of $n$, such that
$a_{ij}(x)\partial_{ij}\theta_n+b_i(x)\partial_i\theta_n>0$
a.e.~in $B_\delta(x_n)\backslash B_\rho(x_n)$.
%
On the other hand,
$$L\theta_n\geq- k'+\left(\inf_{B_\delta(x_n)}c\right)
\left(\min_{[0,\rho]}\vartheta\right)\quad\text{a.e.~in }B_\rho(x_n),$$
where $k'$ is another positive constant independent of $n$.
As a consequence, using the hypothesis on $c$, we can find $n\in\N$ such that
$L\theta_n>0$ a.e.~in $B_\delta(x_n)$.
Taking $\phi=\theta_n$ in \eq{l1'} we obtain
$\l'(-L,B_\delta(x_n))\leq0$. Eventually, statement (i) of Propositions
\ref{l1} and \ref{l1'} yield
$$0\geq\l'(-L,B_\delta(x_n))=\lambda_{B_\delta(x_n)}=
\l(-L,B_\delta(x_n)).$$ 
That the MP does not hold in this case (in fact, as soon as $-\infty\leq\l<0$)
follows from Theorems \ref{thm:MP} and
\ref{thm:relations} part (ii), proved in Sections \ref{sec:MP} and
\ref{sec:relations} respectively.
%
\end{proof}

The hypothesis on $c$ in the previous statement cannot be weakened by
$\sup c=+\infty$. One can see this by considering, in dimension $1$, the
operator $L u:=u''+c(x)u$, with $c(x):=v'(x)-v^2(x)$ and $v\in
C^1(\R)$ such that $\sup(v'-v^2)=+\infty$. We leave to the reader to check that
such a function $v$ exists. Since 
the function $\phi(x)=e^{-\int_0^x v(t)dt}$ satisfies $L\phi=0$, it follows that
$\l(-L,\R)\geq0$.
We now show that if
the $b_i$ are unbounded then it may happen that
$\l(-L,\O)\in\R$ even though $c$ satisfies \eq{supc}.

\begin{proposition}\label{pro:l1inR}
If the $a_{ij}$ are bounded, $b(x)\.x$ does not change sign
for $|x|$ large and it holds that
$$\limsup_{\su{x\in\O}{|x|\to\infty}}\frac{c(x)}{\frac{|b(x)\.x|}{|x|}+1}<+\infty,$$
then $\l(-L,\O)$ is finite.
\end{proposition}

\begin{proof}
Let $\phi\in C^2(\R^N)$ be a positive function satisfying, for $|x|\geq1$,
$\phi(x)=e^{\pm\sigma|x|}$, where the $\pm$
is in agreement with the sign of $-b\.x$ at infinity, and $\sigma>0$ will be
chosen later. Direct computation shows that
$$L\phi=\left[\frac{a_{ij}x_ix_j}{|x|^2}\sigma^2
\pm\left(\frac{\tr(a_{ij})}{|x|}-\frac{a_{ij}x_ix_j}{|x|^3}
+\frac{b\.x}{|x|}\right)\sigma+c\right]\phi
\quad\text{a.e.~in }\O\backslash B_1.$$
Using the hypotheses, we can then choose $\sigma$ large enough and
$\lambda\in\R$ such that $(L+\lambda)\phi<0$ a.e.~in
$\O$. Hence, $\l(-L,\O)\geq\lambda$. On the other hand, $\l(-L,\O)<+\infty$
by Proposition \ref{l1} part (ii). This proof also shows that
$\l''(-L,\O)\in\R$,
under the same conditions, when $b(x)\.x<0$ at infinity.
\end{proof}


\section{Existence of positive eigenfunctions vanishing on $\partial\O$}
\label{sec:charE}

We now prove the characterization of the set of eigenvalues
$\mc{E}$. One of the main tools we require is the following
boundary Harnack inequality, quoted from \cite{BCN3}, which extends the
previous versions of \cite{Caf-bH}, \cite{Bau}. We recall that,
for $\delta>0$, $\O^\delta$ denotes the set $\{x\in\O\ :\
\dist(x,\partial\O)>\delta\}$.

\begin{theorem}[\cite{BCN3}]\label{thm:bHarnack}
Let $\O$ be a bounded domain and $\O'$ be an open subset of $\O$
such that $T:=\partial\O\cap (\O'+B_\eta)$ is of class $C^{1,1}$,
for some $\eta>0$. Then,
any nonnegative solution $u\in W^{2,N}_{loc}(\O)\cap
C^0(\O\cup T)$ of
$$\left\{\begin{array}{ll}
Lu=0 & \text{a.e.~in }\O\\
u=0 & \text{on }T,\\
\end{array}\right.$$
satisfies
$$\sup_{\O'} u\leq C\inf_{\O^\delta}u,$$
for all $\delta>0$ such that $\O^\delta\neq\emptyset$, with $C$ depending on
$N$, $\O$, $\delta$, $\eta$, $\inf\ul\alpha$ and the
$L^\infty$ norms of $a_{ij}$, $b_i$, $c$.
\end{theorem}

\begin{proof}[Proof of \thm{spettro}]
From the definition \eq{l1} of $\l(-L,\O)$ it follows that
$\mc{E}\subset(-\infty,\l(-L,\O)]$. This concludes the proof if
$\l(-L,\O)=-\infty$. Let us prove the reverse inclusion when
$\l(-L,\O)>-\infty$.
We can assume, without loss of generality, that $0\in\O$.
Since $\O$ is smooth, a compactness argument (that we leave to the reader) shows
that, for any
$n\in\N$,
there exists $r(n)\geq n$ such that $\O\cap B_n$ is contained in a single
connected
component of $\O\cap B_{r(n)}$. Let $\O_n$ denote this connected component.
It is not restrictive to assume that $\O_n\subset\O_{n+1}$ for $n\in\N$.
Hence, Proposition \ref{l1} part (iv) yields
$$\limn\l(-L,\O_n)=\l(-L,\O).$$
We first show the existence of an eigenfunction associated with $\l(-L,\O)$,
and then of one associated with $\lambda$, for any given $\lambda<\l(-L,\O)$.

{\em Step 1:\, $\l(-L,\O)\in\mc{E}$.}\\
For $n\in\N$, let $\vp^n$ be the generalized principal
eigenfunction of $-L$ in $\O_n$, normalized by $\vp^n(0)=1$. This eigenfunction
is obtained in the work of Berestycki, Nirenberg and Varadhan \cite{BNV} (note
that $\O_n$, in general, is not smooth). For
the reader's ease, some of the main results of that paper are described in
Appendix \ref{sec:BNV} here. In particular, the existence of $\vp^n$ is provided
by Property
A.1. Fix $m\in\N$. Since for
$n>m$, $\vp^n$ belongs to $W^{2,p}(\O\cap B_m)$, $\fa p<\infty$,
and vanishes on $\partial\O\cap B_m$, applying the boundary
Harnack inequality - \thm{bHarnack} - with
$\O=\O_{m+1}$,
$\O'=\O\cap B_m$, $\eta=1$ and $\delta<\dist(0,\partial\O)$, we
find a constant $C_m$ such that
$$\fa n> m,\quad
\sup_{\O\cap B_m}\vp^n\leq C_m.$$
Thus, the elliptic local boundary estimate of Agmon, Douglis and
Nirenberg \cite{ADN} (see also Theorem 9.13 of \cite{GT}) implies that the
$(\vp^n)_{n>m}$ are uniformly bounded in $W^{2,p}(\O\cap
B_{m-\frac12})$ (since $\ol B_{m-\frac12}\cap\partial\O$ is contained in a
smooth boundary portion of $B_m\cap\O$). Consequently they converge, up to
subsequences, weakly in $ W^{2,p}(\O\cap B_{m-\frac12})$ and, 
by Morrey's inequality (see, e.g., Theorem 7.26 part (ii) of \cite{GT}),
strongly in $C^1(\ol\O\cap B_{m-1})$ to a
nonnegative solution $\phi^m$ of
$$\left\{\begin{array}{ll}
         -L\phi^m=\l(-L,\O)\phi^m & \text{a.e.~in }\O\cap B_{m-1}\\
         \phi^m=0 & \text{on }\partial\O\cap B_{m-1}.
         \end{array}\right.$$
In particular, $\phi^m(0)=1$ and then $\phi^m$ is positive in $\O\cap B_{m-1}$
by the \SMP.
Therefore, using a diagonal method, we can extract a subsequence of
$(\vp^n)_{n\in\N}$ converging to a positive function $\phi$ which is
a solution of the above problem for all $m>1$. That is, $\l(-L,\O)\in\mc{E}$.

{\em Step 2:\, $(-\infty,\l(-L,\O))\subset\mc{E}$.}\\
Take $\lambda<\l(-L,\O)$. Since $\O$ is unbounded and connected,
$\O_n\backslash\ol B_{n-1}\neq\emptyset$ for all $n\in\N$. Let
$(f_n)_{n\in\N}$ be a family of continuous, nonpositive and not identically
equal to zero functions such that
$$\fa n\in\N,\quad \supp f_n\subset\O_n\backslash\ol B_{n-1}.$$
Since for $n\in\N$, $\l(-L,\O_n)>\l(-L,\O)
>\lambda$ by Proposition \ref{l1} part (iii), we have
$\l(-(L+\lambda),\O_n)>0$. Hence, Property A.5 provides a
bounded solution $u^n\in W^{2,N}_{loc}(\O_n\cup(B_n\cap\partial\O))$ of
$$
\left\{\begin{array}{ll}
(L+\lambda)u^n=f_n & \text{a.e.~in }\O_n\\
u^n\uguale0 & \text{on }\partial\O_n.\\
\end{array}\right.
$$
The meaning of the relaxed boundary condition $u^n\uguale0$ is recalled in
Appendix \ref{sec:BNV}. However, we only use here the fact that it implies
$u^n=0$ in the classical sense on the smooth portion $B_n\cap\partial\O$. Note
that $u^n$ is nonnegative by Property A.2, and then it is
strictly positive in $\O_n$ by the \SMP. Moreover, Lemma 9.16 in
\cite{GT} yields $u^n\in W^{2,p}_{loc}(\O_n\cup(B_n\cap\partial\O))$, $\fa
p<\infty$. For $n\in\N$, the function $v^n$ defined
by
$$v^n(x):=\frac{u^n(x)}{u^n(0)},$$
belongs to $W^{2,p}(\O\cap B_{n-1})$, it is positive and satisfies: $v^n(0)=1$,
$$
\left\{\begin{array}{ll}
-Lv^n=\lambda v^n & \text{a.e.~in }\O\cap B_{n-2}\\
v^n=0 & \text{on }\partial\O\cap B_n.\\
\end{array}\right.
$$
We can thereby proceed exactly as in step 1, with $(v^n)_{n\in\N}$
in place of $(\vp^n)_{n\in\N}$, and infer that $\lambda\in\mc{E}$.
\end{proof}

\begin{remark}\label{rem:portion}
Actually, the arguments in the proof of \thm{spettro} yield a more general
statement. Namely, if $\O$ is unbounded and has a smooth boundary portion
$T$
then $\mc{E}_T=(-\infty,\l(-L,\O)]$, where
\[\begin{split}
\mc{E}_T:=\{\lambda\in\R\ :\ \ex\phi\in W^{2,p}_\loc(\O\cup T),\ \fa p<\infty,\
\phi>0,
\ -L\phi=\lambda\phi\text{ a.e.~in }\O,\\
\phi=0\text{ on }T\}.
\end{split}\]
\end{remark}

We now exhibit an example where the set of eigenvalues does not reduce to
$\{\l(-L,\O)\}$ even if one restricts to (positive)
eigenfunctions decaying to $0$ at infinity.
This example also
shows that $\l>0$ does not imply the validity of
the MP for subsolutions which are nonpositive also at infinity, and not only on
$\partial\O$.

\begin{counter}\label{c-e}
There exists an operator $L$ in $\R$ such that
$\l(-L,\R)>0$ and, for all $\lambda\in[0,\l(-L,\R)]$, there is a positive
function $\phi\in W^{2,p}(\R)$, $\fa p<\infty$, satisfying
$$-L\phi=\lambda\phi\quad\text{a.e.~in }\R,\qquad
\limsup_{|x|\to\infty}\phi(x)e^{|x|}\leq1.$$
\end{counter}

\begin{proof}
Consider the operator $L$ defined by
$$Lu(x):=\left\{\begin{array}{ll}
u''(x)-4u'(x)+3u(x) & \text{if }x<-\frac\pi4\\
u''(x)+u(x) & \text{if }-\frac\pi4\leq x\leq\frac\pi4\\
u''(x)+4u'(x)+3u(x) & \text{if }x>\frac\pi4.
\end{array}\right.$$
In order to show that $\l(-L,\R)>0$, we explicitly
construct a function $v\in W^{2,\infty}(\R)$ such that
$(L+\lambda)v\leq0$,
for some $\lambda>0$. We set
$$v(x):=\left\{\begin{array}{ll}
ke^{2x} & \text{if }x<-\frac\pi4\\
\cos(\gamma x) & \text{if }-\frac\pi4\leq x\leq\frac\pi4\\
ke^{-2x} & \text{if }x>\frac\pi4,
\end{array}\right.$$
where $k=e^{\frac\pi2}\cos(\frac\pi4\gamma)$ and $\gamma$ is the solution in
$(1,2)$ of the equation 
$$\gamma\tan\left(\frac\pi4\gamma\right)-2=0.$$
We leave to the reader to check that $v\in W^{2,\infty}(\R)$. We see that
$Lv=-v$
for $|x|>\pi/4$. For $|x|<\pi/4$, we find $Lv=(1-\gamma^2)v$.
Hence, $(L+\lambda)v\leq0$ a.e.~in $\R$, with
$\lambda=\min(1,\gamma^2-1)>0$. 
Now, direct computation shows that the function
$$\ul u(x):=\left\{\begin{array}{ll}
e^x & \text{if }x<-\frac\pi4\\
\sqrt2e^{-\frac\pi4}\cos(x) & \text{if }-\frac\pi4\leq x\leq\frac\pi4\\
e^{-x} & \text{if }x>\frac\pi4
\end{array}\right.$$
belongs to $W^{2,\infty}(\R)$ and satisfies $L\ul u=0$ in
$\R\backslash\{\pm\pi/4\}$.
For $\lambda\in[0,\l(-L,\R)]$, let $\phi$ be the associated positive
eigenfunction
constructed as in the proof of \thm{spettro}, with $\O_n=B_n$.
It is clear that, when $\lambda<\l(-L,\R)$,
it is possible to take an even function $f_n$ in that construction.
Hence, the symmetry of $L$ implies that $\phi$ is even.
Normalize $\phi$ in such a way that $\phi(0)<\ul u(0)$.
By property (iv) of Proposition \ref{l1}, $\l(-L,B_r)>0$ for $r$ large enough.
Thus, if $\ul u(\pm r)\leq\phi(\pm r)$ for such values of $r$,
the MP yields a contradiction. This shows that $\phi(x)<\ul u(x)$
for $|x|$ large enough, which concludes the proof.

We remark that $\l'(-L,\R)\leq-1$, as is seen by taking $\phi\equiv1$ in
\eq{l1'}. This is in agreement with \thm{MP}.
\end{proof}


\section{Maximum principle}\label{sec:MP}

We derive \thm{MP} as a particular case of a result concerning
subsolutions bounded from above by (constant times) a barrier $\beta$.
The function $\beta$ is positive and satisfies either
\Fi{b<poli}
\ex\sigma>0,\quad\limsup_{\su{x\in\O}{|x|\to\infty}}\beta(x)|x|^{-\sigma}=0,
\Ff
if the coefficients of $L$ satisfy \eq{ABC3}, or
\Fi{b<exp}
\ex\sigma>0,\quad\limsup_{\su{x\in\O}{|x|\to\infty}}\beta(x)e^{-\sigma|x|}=0,
\Ff
if they satisfy the stronger hypothesis
\Fi{sub1}
\sup_\O c<\infty,\qquad
\sup_\O a_{ij}<\infty,\qquad
\sup_{x\in\O}\frac{b(x)\.x}{|x|}<\infty.
\Ff

\begin{definition}\label{def:bMP}
Let $\beta$ be a positive function on $\O$.
We say that the operator $L$ satisfies the {\em $\beta$-MP} in $\O$ if every
function $u\in W^{2,N}_\loc(\O)$ such that
$$Lu\geq0 \ \text{ a.e.~in }\O,\qquad
\sup_{\O}\frac u\beta<\infty,\qquad
\fa\xi\in\partial\O,\ \ \limsup_{x\to\xi}u(x)\leq0,$$
satisfies $u\leq0$ in $\O$.
\end{definition}

\thm{MP} represents the particular case $\beta\equiv1$ of the following
statement.

\begin{theorem}\label{thm:bMP}
The operator $L$ satisfies the $\beta$-MP in $\O$
\begin{enumerate}[(i)]
\item if $\lambda_\beta''(-L,\O)>0$ and either \eq{ABC3}, \eq{b<poli}
or
\eq{sub1}, \eq{b<exp} hold;

\item only if $\;\lambda_\beta'(-L,\O)\geq0$.
\end{enumerate}
\end{theorem}

\begin{proof}
Statement (ii) is an immediate consequence of Definition \ref{def:l1b}. Indeed, if
$\lambda_\beta'(-L,\O)<0$ then  
there are $\lambda<0$ and a positive function $\phi\in W^{2,N}_\loc(\O)$ such
that
$$\phi\leq\beta,\quad
L\phi\geq-\lambda\phi \ \text{ a.e.~in }\O,\qquad
\fa\xi\in\partial\O,\ \ \lim_{x\to\xi}\phi(x)\leq0.$$
Hence, $\phi$ violates the $\beta$-MP.

Let us prove (i). Assume by contradiction that there exists a function $u\in
W^{2,N}_\loc(\O)$ which is positive somewhere in $\O$ and satisfies
$$Lu\geq0 \ \text{ a.e.~in }\O,\qquad
\sup_{\O}\frac u\beta<\infty,\qquad
\fa\xi\in\partial\O,\ \ \limsup_{x\to\xi}u(x)\leq0.$$
Since $\lambda_\beta''(-L,\O)>0$, by Definition \ref{def:l1b} there exists
$\lambda>0$ and a function
$\phi\in W^{2,N}_\loc(\O)$ such that
$$\phi\geq\beta,\qquad (L+\lambda)\phi\leq0,\quad\text{a.e.~in }\O.$$
In particular, up to renormalization, we can assume that $\phi\geq u$ in $\O$.
We want to modify $\phi$ in order to obtain a function that
grows faster than $u$ at infinity and is still a
supersolution in a suitable subset of $\O$.
To this aim, we consider a positive smooth function $\chi:\R^N\to\R$
such that, for $|x|>1$, $\chi(x)=|x|^\sigma$ if $\beta$ satisfies \eq{b<poli}
or $\chi(x)=e^{\sigma|x|}$ if $\beta$ satisfies \eq{b<exp}.
For $n\in\N$, we set
$$\phi_n(x):=\phi(x)+\frac1n\chi(x),\qquad
k_n:=\sup_{\O}\frac u{\phi_n}.$$
Note that the sequence $\seq{k}$ is positive, nondecreasing
and bounded from above by $1$.
Thus,
it is convergent. Moreover, since
$$\limsup_{\su{x\in\O}{|x|\to\infty}}\frac{u(x)}{\phi_n(x)}\leq
n\left(\sup_\O\frac u\beta\right)\limsup_{\su{x\in\O}{|x|\to\infty}}
\frac{\beta(x)}{\chi(x)}=0,\qquad\fa\xi\in\partial\O,\quad\limsup_{x\to\xi}\frac
{ u(x) }{\phi_n(x)}=0,$$
there exists $x_n\in\O$ such that $k_n=\frac{u(x_n)}{\phi_n(x_n)}$.
We claim that, for $n$ large enough, $L\phi_n<0$
in a neighborhood of $x_n$.
The operator $L$ acts on a radial function $\theta(x)=\vartheta(|x|)$ in the
following way:
$$L\theta(x)=A(x)\vartheta''(|x|)+B(x)\vartheta'(|x|)+c(x)\vartheta(|x|),$$
where
\Fi{AB} A(x):=\frac{a_{ij}(x)x_ix_j}{|x|^2},\qquad
B(x):=\frac{b(x)\.x}{|x|}+ \frac{\tr
(a_{ij}(x))}{|x|}-\frac{a_{ij}(x)x_ix_j}{|x|^3}. \Ff
Hence, for a.e.~$x\in\O\backslash B_1$, in the case where $\beta$
satisfies \eq{b<poli} we get
\[\begin{split}
L\chi &=\left(\sigma(\sigma-1)\frac{A(x)}{|x|^2}+\sigma\frac{B(x)}{|x|}
+c(x)\right)\chi\\
&\leq\left(\sigma(N+\sigma-2)\frac{\ol\alpha(x)}{|x|^2}
+\sigma \frac{b(x)\.x}{|x|^2}+c(x)\right)\chi,
\end{split}\]
while, in the case of condition \eq{b<exp}, we get
\[\begin{split}
L\chi &=(\sigma^2A(x)+\sigma B(x)+c(x))\chi\\
&\leq\left[\sigma\left(\sigma+\frac{N-1}{|x|}\right)\ol\alpha(x)
+\sigma\frac{b(x)\.x}{|x|}+c(x)\right]\chi.
\end{split}\]
Therefore, in both cases, there exists a positive constant $C$ such that
$L\chi\leq C\chi$ a.e.~in $\O$.
Let us estimate the ``penalization'' term
$\frac1n\chi(x_n)$.
For $n\in\N$, we find that
$$\frac1{k_{2n}}\leq\frac{\phi_{2n}(x_n)}{u(x_n)}=
\frac{\phi(x_n)+\frac1{2n}\chi(x_n)}{u(x_n)}=
\frac1{k_n}-\frac{\chi(x_n)}{2nu(x_n)},$$
and then that
$$\frac{\chi(x_n)}n\leq 2\left(\frac1{k_n}-\frac1{k_{2n}}\right)u(x_n).$$
As a consequence, for $n\in\N$, there exists $\delta_n>0$ such that, for
a.e.~$x\in B_{\delta_n}(x_n)$,
$$
\frac1n L\chi(x)\leq C\frac{\chi(x)}n\leq
3C\left(\frac1{k_n}-\frac1{k_{2n}}\right)u(x)\leq
3C\left(\frac1{k_n}-\frac1{k_{2n}}\right)\phi(x).
$$
Thus,
$$L\phi_n\leq\left[-\lambda+3C\left(\frac1{k_n}-\frac1{k_{2n}}\right)\right]\phi
\quad\text{a.e.~in }B_{\delta_n}(x_n).$$
Since the sequence $\seq{k}$ is convergent, we
can then find $n\in\N$ such that $L\phi_n<0$ a.e.~in $B_{\delta_n}(x_n)$.
Whence we infer that the nonnegative function $w_n:=k_n\phi_n-u$
satisfies $Lw_n<0$ a.e.~in $B_{\delta_n}(x_n)$ and vanishes at $x_n$. This
contradicts the \SMP.
\end{proof}

%
%
%
%
%
%
%
%
%
%
%

\begin{remark}\label{rem:MP0}
If $\l'(-L,\O)=\l''(-L,\O)=0$ then the MP might or might not hold. Indeed,
if $L$ and $\O$ are periodic then $\l'(-L,\O)$ and $\l''(-L,\O)$ coincide with
$\lambda_p$.
Hence, if $\lambda_p=0$, the periodic principal eigenfunction violates the MP.
On the other hand,
the operator $L$ introduced at the beginning of the proof of Proposition
\ref{pro:nopef} below
satisfies the MP and $\l'(-L,\R)=\l''(-L,\R)=0$.
\end{remark}


\section{Properties of $\l''$}\label{sec:l1''}

\begin{proposition}\label{pro:l1''}
The quantity $\l''(-L,\O)$ defined by \eq{l1''} satisfies the following
properties:
\begin{enumerate}[(i)]
\item if $\O$ is bounded and smooth then $\l''(-L,\O)=\lambda_\O$;

\item $$-\sup_\O c\leq\l''(-L,\O)\leq\l(-L,\O);$$

\item if $\O'\subset\O$ then
$\l''(-L,\O')\geq\l''(-L,\O)$;

\item if $\O$ is smooth and $\O,\ L$ are periodic, with the same
period, then $\l''(-L,\O)=\lambda_p$;

\item
in its dependence on $c$, $\l''(-L,\O)$ is nonincreasing,
concave and Lipschitz-continuous (using the $L^\infty$ norm) with
Lipschitz constant 1;

\item
for uniformly elliptic operators with bounded coefficients,
$\l''(-L,\O)$ is locally Lipschitz-continuous with respect to the
$b_i$, with Lipschitz constant depending only on $N,\ \O$, the
ellipticity constants and the $L^\infty$ norm of $c$.
\end{enumerate}
\end{proposition}

\begin{proof}
The first inequality in property (ii) follows by taking $\phi\equiv1$ in
\eq{l1''}.
The second inequality in (ii), as well as property (iii), are immediate
consequences of the
definition.

(i) From (ii) and Proposition \ref{l1} part (i) it follows that
$\l''(-L,\O)\leq\lambda_\O$. The reverse inequality is a consequence of 
Lemma \ref{lem:infT>0}. Note that if $\O$ is of class $C^{2,1}$ then this
inequality also follows from the
characterization of Proposition \ref{pro:l1''=}, taking $\phi$ 
equal to the Dirichlet principal eigenfunction of $-L$ in $\O$.

(iv) We consider the periodic principal eigenfunction $\vp$ of $-L$ in $\O$,
under Dirichlet boundary conditions. Taking $\phi=\vp$ in the
characterization of Proposition \ref{pro:l1''=} yields
$\l''(-L,\O)\geq\lambda_p$.
Assume now by contradiction that $\l''(-L,\O)>\lambda_p$. Thus, by \thm{MP},
the operator $(L+\lambda_p)$ satisfies the MP in $\O$. This is in contradiction
with the existence of the periodic principal eigenfunction.

Properties (v), (vi) follow from the same arguments used to prove the analogous
properties for $\l$ (cf.~Propositions 2.1, 5.1 of \cite{BNV}).
\end{proof}

We now derive a result about the admissible functions $\phi$ in \eq{l1''}. It 
will be used in the sequel to obtain the sufficient conditions for the
equivalence of $\l$, $\l'$, $\l''$.

\begin{proposition}\label{pro:C1}
If $\O$ has a $C^{1,1}$ boundary portion $T\subset\partial\O$ 
then the definition \eq{l1''} of $\l''(-L,\O)$
does not change if one further requires
$\phi\in W^{2,p}_\loc(\O\cup T)$, $\fa p<\infty$.
\end{proposition}

\begin{proof}
We prove the statement by showing that, if for some $\lambda\in\R$ there exists
a function $\phi\in W^{2,N}_\loc(\O)$ satisfying
$$\inf_\O\phi>0,\qquad (L+\lambda)\phi\leq0\quad\text{a.e.~in }\O,$$
then we can find a function $u\in W^{2,p}_\loc(\O\cup T)$, $\fa p<\infty$,
with the same properties. The function $u$
will be obtained as a solution of a suitable nonlinear problem. 

First, by
renormalizing $\phi$ and replacing $c$ with $c+\lambda$, the problem is reduced
to the case where $\inf_\O\phi=2$ and $\lambda=0$. Consider the
function $f:\O\times\R\to\R$ defined by $f(x,s):=|c(x)|g(s)$,
where
$$g(s)=\left\{\begin{array}{ll}
-1 & \text{for }s\leq1\\
s-2 &  \text{for }s\in(1,2)\\
0 & \text{for }s\geq2.
               \end{array}\right.$$
Setting $\ul u\equiv1$, we find that, a.e.~in $\O$,
$$L \ul u\geq f(x,\ul u),\qquad
L\phi\leq f(x,\phi),\qquad \ul u<\phi.$$ Standard
arguments provide a function $u\in W^{2,p}_{loc}(\O)$
satisfying $1\leq u\leq\phi$ and $Lu=f(x,u)$ a.e.~in $\O$. 
More precisely, one constructs solutions of problems in bounded domains
invading $\O$ by an iterative method and then uses a diagonal extraction
procedure. 
However, getting the improved
regularity $u\in W^{2,p}_{loc}(\O\cup T)$
is delicate, especially because $\phi$ may blow up at $T$. 
Moreover, in order to pass to the unbounded domain, one needs a
version of the boundary Harnack inequality for solutions of
inhomogeneous problems. This is the
object of Appendix \ref{sec:bHarnack}. Let us now describe the method in detail.

The first step consists in solving semilinear problems in bounded 
domains with Dirichlet conditions on smooth portions of the boundary.
Namely, we derive the following

\begin{lemma}\label{lem:semilinear}
Let $\O$ be a bounded domain with a $C^{1,1}$ boundary portion $T$ and
let $f:\O\times\R\to\R$ be such that $f(\.,0)\in L^N(\O)$ and
$f(x,\.)$ is uniformly Lipschitz continuous in $\R$, uniformly with respect to
$x\in\O$.
Assume further that the problem
$$\left\{\begin{array}{ll}
Lu=f(x,u) & \text{a.e.~in }\O\\
u\uguale0 & \text{on }\partial\O
\end{array}\right.$$
has a subsolution $\ul u\in W^{2,N}_{loc}(\O)\cap L^\infty(\O)$ and a
supersolution 
$\ol u\in W^{2,N}_{loc}(\O)$ such that
$\ul u\leq0\leq\ol u$ in $\O$.
Then, there exists a function $u\in W^{2,N}_{loc}(\O\cup T)\cap L^\infty(\O)$
satisfying
$$\left\{\begin{array}{ll}
Lu=f(x,u) & \text{a.e.~in }\O\\
u=0 & \text{on }T\\
\ul u\leq u\leq\ol u & \text{in }\O.
\end{array}\right.$$
\end{lemma}

Note that, in the above statement, one can replace the 0 boundary conditions
with a more general datum $\psi\in W^{2,N}(\O)$, provided that $\ul u$, $\ol u$
satisfy $\ul u\leq\psi\leq\ol u$.
Let us postpone the proof of this Lemma until we complete the argument
to prove Proposition \ref{pro:C1}.
We assume that $0\in\O$. For $n\in\N$, let $\O_n$
denote the connected component of $\O\cap B_n$ containing $0$ and let $T_n$
be its portion of the boundary of class at least $C^{1,1}$. $T_n$ is open in the
topology
of $\partial\O_n$ and is nonempty for $n$ large enough.
Consider the functions $u^n\in W^{2,N}_{loc}(\O_n\cup T_n)\cap L^\infty(\O_n)$ 
provided by Lemma \ref{lem:semilinear} satisfying
$$\left\{\begin{array}{ll}
Lu^n=f(x,u^n) & \text{a.e.~in }\O_n\\
u^n=1 & \text{on }T_n\\
1\leq u^n\leq\phi& \text{in }\O_n.
\end{array}\right.$$
Since $|f(x,u^n)|\leq|c(x)|$ and $1\leq u^n\leq\phi$, 
using interior estimates and a diagonal argument, we find that the 
sequence $(u^n)_{n\in\N}$ converges
(up to subsequences) locally uniformly in $\O$ to a function $u
\in W^{2,N}_{loc}(\O)$ satisfying
$$Lu=f(x,u)\leq0\quad\text{a.e.~in }\O,\qquad 1\leq u\leq\phi\quad\text{in
}\O.$$ 
It remains to show that $u\in W^{2,p}_{loc}(\O\cup T)$, $\fa
p<\infty$. Let $K\subset\subset\O\cup T$.
The smoothness of $T$ implies that $K\subset\subset\O_m\cup T_m$, for $m$ large
enough. 
Take $\eta>0$ such that $\partial\O_m\cap(K+B_{2\eta})\subset T_m$.
Applying the inhomogeneous boundary Harnack inequality given by Proposition
\ref{pro:bHarnack}, with $\O=\O_m$ and $\O'=\O_m\cap(K+ B_\eta)$,
we find a constant $C$ such that
$$\fa n\geq m,\quad\sup_{\O_m\cap(K+ B_\eta)}u^n\leq
C(u^n(0)+1)\leq
C(\phi(0)+1).$$
Hence, by the local boundary estimate, $(u^n)_{n\in\N}$ is bounded in
$W^{2,p}(K)$ 
and then its limit $u$ belongs to $W^{2,p}(K)$.
\end{proof}

\begin{proof}[Proof of Lemma \ref{lem:semilinear}]
Replacing $c$ with $c-k$ and $f(x,s)$ with $f(x,s)-ks$ if need be, with $k$
greater than 
$\|c\|_{L^\infty(\O)}$ and the 
Lipschitz constant of $f(x,\.)$, it is not restrictive to assume that $c$ is
negative and that
$f(x,\.)$ is decreasing. From Proposition \ref{l1} parts (ii) and (iii)
it follows that $\l(-L,\mc{O})>0$ in any bounded domain
$\mc{O}$.
Hence, by Property A.5 in Appendix \ref{sec:BNV}, the problem
$$\left\{\begin{array}{ll}
Lu^1=f(x,\ul u) & \text{a.e.~in }\O\\
u^1\uguale0 & \text{on }\partial\O
\end{array}\right.$$
admits a unique bounded solution $u^1\in W^{2,N}_\loc(\O\cup T)$ (note that
$f(x,\ul u)\in L^N(\O))$. 
The function $\ul u$ is a subsolution of this problem and $\ol u$ is 
a supersolution by the monotonicity of $f(x,\.)$.
By the refined MP - Property A.2 - we get $\ul u\leq u^1$,
but we cannot infer that $u^1\leq\ol u$, because $\ol u$ may be unbounded.
However, since the solution $u^1$ is obtained as the limit of
solutions $(u^1_n)_{n\in\N}$ of the Dirichlet problem in a family of bounded
smooth domains invading $\O$ (see the proof of Theorem 1.2 in \cite{BNV}), 
the inequality $u^1\leq\ol u$ follows by applying the refined MP to the
functions
$\ol u-u^1_n$.
Proceeding as before, we construct by iteration a sequence $(u^j)_{j\in\N}$
in $W^{2,N}_\loc(\O\cup T)\cap L^\infty(\O)$
such that
$$\left\{\begin{array}{ll}
Lu^{j+1}=f(x,u^j) & \text{a.e.~in }\O\\
u^{j+1}\uguale0 & \text{on }\partial\O\\
u^j\leq u^{j+1}\leq\ol u & \text{in }\O.
\end{array}\right.$$
For $x\in\O$, let $u(x)$ be the limit of the nondecreasing sequence
$(u^j(x))_{j\in\N}$.
Let us show that the $u^j$ are uniformly bounded in $\O$.
We write 
$$Lu^j=f(x,0)+\zeta_j(x)u^j\quad
\text{a.e.~in }\O,\qquad\text{with }\ \zeta_j(x):=
\frac{f(x,u^j)-f(x,0)}{u^j}.$$
Since the $L^\infty$ norm of the $\zeta_j$ is less than or equal to the
Lipschitz constant of
$f(x,\.)$, applying the ABP estimate - Property A.6 - to $u^j$ and $-u^j$ we
infer
that $(u^j)_{j\in\N}$ is bounded in $L^\infty(\O)$. Therefore, by the local
boundary estimate,
$(u^j)_{j\in\N}$ is bounded in $W^{2,N}(K)$, for any $K\subset\subset\O\cup T$.
Whence, considering suitable subsequences of $(u^j)_{j\in\N}$ and applying the
embedding theorem, 
we derive $u\in W^{2,N}_\loc(\O\cup T)\cap L^\infty(\O)$, $Lu=f(x,u)$ a.e.~in
$\O$ 
and $u=0$ on $T$. This concludes the proof.
\end{proof}


\section{Relations between $\l,\ \l'$ and $\l''$}
\label{sec:relations}

This section is devoted to the proof of \thm{relations}. We
will start from statement (ii). In our previous work \cite{BR1}, we proved it in
dimension 1, using a direct argument, and we left the case
of arbitrary dimension as an open problem. Here, we solve it by subtracting
a quadratic penalization term that prevents solutions from being unbounded.

\begin{proof}[Proof of \thm{relations} part (ii)]
We prove the statement by showing that, for any given
$\lambda>\l(-L,\O)$, $\l'(-L,\O)\leq\lambda$. We can
assume without loss of generality that $\lambda=0$.
Since $\l(-L,\O)<0$, by Proposition \ref{l1} part (iv) there exists a bounded
smooth domain $\O'\subset\O$ such that $\lambda_{\O'}<0$. Let
$\vp'$ be the principal eigenfunction associated with
$\lambda_{\O'}$, normalized by
$$\|\vp'\|_{L^\infty(\O')}=\min\left(1,
-\frac{\lambda_{\O'}}{\|c\|_{L^\infty(\O')}}\right).$$ Then, the
functions $\ul u$, $\ol u$ defined by
$$\ul u(x):=\left\{\begin{array}{ll}\vp'(x) & \text{if }x\in\O'\\
0 & \text{otherwise},\end{array}\right.\qquad \ol u(x) :=1,$$
satisfy, a.e.~in $\O$,
$$L\ul u\geq c^+(x)\ul u^2,\qquad L\ol u\leq c^+(x)\ol u^2,
\qquad\ul u\leq\ol u,$$ 
where $c^+(x)=\max(c(x),0)$. Thus, there exists a solution $u\in
W^{2,p}_{loc}(\ol\O)$, $\fa p<\infty$, of the problem
$$
\left\{\begin{array}{ll}Lu=c^+(x)u^2 & \text{a.e.~in }\O\\
u=0 & \text{on }\partial\O,\end{array}\right.
$$
such that $\ul u\leq u\leq\ol u$ in $\O$ (note that $\ul u$ is a
``generalized subsolution'' of the above problem because it is the
supremum of two subsolutions). 
The existence of $u$ follows from the same arguments as 
in the proof of Proposition \ref{pro:C1}, but here is actually simpler 
because the supersolution $\ol u$ is bounded.
In particular, we see that $(L-c^+(x))u\leq0$
a.e.~in $\O$ and then the \SMP\ yields $u>0$ in $\O$. Taking
$\phi=u$ in \eq{l1'} we eventually derive $\l'(-L,\O)\leq0$.
\end{proof}

There is also a more direct, linear proof of \thm{relations} part (ii)
\footnote{The authors are grateful to a referee for suggesting this 
approach.}. The arguments, that we sketch now, make use of two 
independent results proved later on in this paper.
As before, the aim is to show that $\l(-L,\O)<0$ implies $\l'(-L,\O)\leq0$. 
Let $(L_n)_{n\in\N}$ be the following family of operators: 
$$L_n=a_{ij}(x)\partial_{ij}+ b_i(x)\partial_i+c_n(x),\qquad\text{with }
c_n(x):=\begin{cases}
           c(x) & \text{if }|x|<n\\
           \min(c(x),0) & \text{otherwise}.
          \end{cases}$$
Note that $\l(-L_n,\O)>-\infty$ by Proposition \ref{l1} part (ii)
and thus \thm{spettro} implies that a principal eigenfunction $\vp_1^n$ 
of $-L_n$ in $\O$ (satisfying the Dirichlet boundary condition) does exist.
Since $\l(-L,\O)<0$, it follows from Proposition \ref{thm:l1C} part (i) that 
$\l(-L_n,\O)<0$ for $n$ large enough.
Applying Proposition \ref{pro:c<0} part (ii) we deduce that $\vp_1^n$ is bounded 
for such values of $n$. Moreover, it satisfies
$$-L\vp_1^n\leq-L_n\vp_1^n=\l(-L_n,\O)\vp_1^n<0\quad\text{a.e.~in }\O.$$
We eventually infer that $\l'(-L,\O)\leq0$.

\begin{remark}
As a byproduct of the above proofs of \thm{relations} part (ii), we have shown 
that the set in definition \eq{l1'} is nonempty when $\O$ is smooth. If in
addition $c$ is bounded from above, the definition of $\l'(-L,\O)$
does not change if one restricts to subsolutions $\phi$ belonging
to $W^{2,p}_{loc}(\ol\O)$, $\fa p<\infty$. Indeed, if $\lambda,\
\phi$ satisfy the conditions in \eq{l1'}, then for any
$\t\lambda>\lambda$ we can argue as in the first proof, with $\ul u=\e\phi$,
$\e$ small enough, and find a positive bounded subsolution of the
Dirichlet problem for $L+\t\lambda$ in $\O$ satisfying the
stronger regularity conditions.
\end{remark}

\begin{proof}[Proof of \thm{relations} part (iii)]
Suppose that there exists $\lambda<\l''(-L,\O)$ and
$\phi\in W^{2,N}_\loc(\O)\cap L^\infty(\O)$ such that
 $$(L+\lambda)\phi\geq0\quad\text{a.e.~in }\O,
 \qquad \fa\xi\in\partial\O,\ \ \limsup_{x\to\xi}\phi(x)\leq0.$$
 Since
$$\l''(-(L+\lambda),\O)=\l''(-L,\O)-\lambda>0,$$
we know from \thm{MP} part (i) that the MP holds for the operator $(L+\lambda)$
in $\O$. As a consequence, $\phi\leq0$ in $\O$. This shows that
$\l'(-L,\O)\geq\l''(-L,\O)$.
\end{proof}


 \begin{proof}[Proof of \thm{relations} part (i)]
 Owing to statement (ii), we only need to prove that $\l(-L,\O)\leq\l'(-L,\O)$.
 That is, if $\lambda\in\R$ and $\phi\in
 W^{2,N}_\loc(\O)\cap L^\infty(\O)$ are such that
 $$\phi>0 \ \text{ in }\O,\qquad(L+\lambda)\phi\geq0 \ \text{ a.e.~in }\O,
 \qquad\fa\xi\in\partial\O,\ \ \lim_{x\to\xi}\phi(x)=0,$$
 then $\l(-L,\O)\leq\lambda$. This will be achieved by the use of the
variational formula \eq{RR}. Clearly, the infimum in \eq{RR} can be taken over
functions in
$H^1_0(\O)$ with compact support in $\ol\O$. Note, however, that since no
restriction is
imposed on the behavior of $c$ at infinity, one cannot consider
the whole space $H^1_0(\O)$.
Let $(\chi_r)_{r>1}$ be a family of cutoff functions
  uniformly bounded in $W^{1,\infty}(\R^N)$ and such that
 $$\fa r>1,\qquad\supp \chi_r\subset B_r,\qquad
 \chi_r=1\ \text{ in }B_{r-1}.$$ We can suppose that $\O\cap B_1\neq\emptyset$.
 The functions
 $\phi\chi_r$ belong to $H^1_0(\O\cap B_r)$.
%
%
 Thus, for $r>1$, we get
 \[\begin{split}
 \l(-L,\O) &\leq \frac{\int_{\O}\left[a_{ij}(x)\partial_i
 (\phi\chi_r)\partial_j(\phi\chi_r)- c(x)\phi^2\chi_r^2\right]}
 {\int_{\O}\phi^2\chi_r^2}\\
 &= \frac{\int_{\O}\left[a_{ij}(x)(\partial_i
 \phi)\chi_r\partial_j(\phi\chi_r)+a_{ij}(x)\phi(\partial_i
 \chi_r)\partial_j(\phi\chi_r)- c(x)\phi^2\chi_r^2\right]}
 {\int_{\O}\phi^2\chi_r^2}.
 \end{split}\]
 Integrating by parts the first term of the above sum yields
 \[\begin{split}
 \l(-L,\O) &\leq \frac{\int_{\O}\left[(-L\phi)\phi\chi_r^2
 -a_{ij}(x)(\partial_i\phi)(\partial_j\chi_r)\phi\chi_r
 +a_{ij}(x)\phi(\partial_i\chi_r)\partial_j(\phi\chi_r)\right]}
 {\int_{\O}\phi^2\chi_r^2}\\
 &\leq \lambda+\frac{ \int_{\O}a_{ij}(x)(\partial_i\chi_r)
 (\partial_j\chi_r)\phi^2} {\int_{\O}\phi^2\chi_r^2}.
 \end{split}\]
 Since $\chi_r=0$ outside $B_r$ and $\chi_r=1$ in
 $B_{r-1}$, we can then find a constant $k>0$, only depending on
 $\sup_{r>1}\|\chi_r\|_{W^{1,\infty}(\R^N)}$ and $\|a_{ij}\|_{L^\infty(\O)}$,
such that
 $$\fa r>1,\quad\l(-L,\O)\leq \lambda+k\frac{\int_{\O\cap(B_r\backslash
 B_{r-1})}
 \phi^2}{\int_{\O\cap B_{r-1}}\phi^2}.$$
 We obtain the desired inequality
 $\l(-L,\O)\leq \lambda$ from the above formula by showing that
$$\liminf_{r\to\infty}\frac{\int_{\O\cap(B_r\backslash
 B_{r-1})} \phi^2}{\int_{\O\cap B_{r-1}}\phi^2}=0.$$
Suppose by contradiction that there exists $\e>0$ such that
$$\fa n>2,\quad\frac{\int_{\O\cap(B_n\backslash
 B_{n-1})} \phi^2}{\int_{\O\cap B_{n-1}}\phi^2}\geq\e.$$
Hence, the sequence $j_n:=\int_{\O\cap B_{n-1}}\phi^2$ satisfies
$j_{n+1}-j_n\geq\e j_n$, that is, $j_n\geq j_2(1+\e)^{n-2}$.
This is impossible because $j_n$ grows at most at the rate $n^N$ as
$n\to\infty$.
 \end{proof}

\begin{remark}\label{rem:lb'=l}
 The previous proof shows that \thm{relations} part (i) holds, more in general,
 with $\l'(-L,\O)$ replaced by $\lambda_\beta'(-L,\O)$ (given by Definition
 \ref{def:l1b}) provided that 
 $$\fa \sigma>0,\quad
 \lim_{\su{x\in\O}{|x|\to\infty}}\beta(x)e^{-\sigma|x|}=0.$$
 On the other hand, if
 $\beta(x)=e^{\sigma|x|}$, $\sigma>0$,
 then one can check that the operator $Lu=u''$ satisfies
 $\l(-L,\R)=0>-\sigma^2=\lambda_\beta'(-L,\R)=\lambda_\beta''(-L,\R)$.
\end{remark}


\section{Conditions for the equivalence of the three notions}
\label{sec:equivalence}


\subsection{Proof of \thm{SC}, case 1}

We start with a preliminary consideration.

\begin{lemma}\label{lem:estimate}
Let $\O$ be bounded and $L$ be self-adjoint. If
$u\in H^1(\O)$, $\chi\in C^1(\ol\O)$ satisfy $u^+\chi\in H^1_0(\O)$,
$Lu\geq0$ in $\O$, then
$$\frac{\l(-L,\O)}{\max_{\ol\O}\ol\alpha}
\int_\O(u^+\chi)^2\leq\int_\O(u^+|\nabla\chi|)^2.$$
\end{lemma}

\begin{proof}
The result follows from a well known inequality which is an immediate 
consequence of the divergence theorem (see, e.g., \cite{DS84}, \cite{Pinch07}). 
Since $Lu\geq0$, we derive
\[\begin{split}
0 &\geq\int_\O a_{ij}(x)\partial_j u\partial_i(u^+\chi^2)-c(x)u u^+\chi^2\\
&= \int_\O a_{ij}(x)[\partial_j(u^+\chi)\partial_i(u^+\chi)
-u^+\partial_j\chi\partial_i(u^+\chi)+u^+\chi\partial_j u^+\partial_i\chi]
-c(x)(u^+\chi)^2\\
&= \int_\O a_{ij}(x)\partial_j(u^+\chi)\partial_i(u^+\chi)-c(x)(u^+\chi)^2
-(u^+)^2a_{ij}(x)\partial_i\chi\partial_j\chi\\
&\geq\l(-L,\O)\int_\O(u^+\chi)^2-\left(\max_{\ol\O}
\ol\alpha\right)\int_\O(u^+|\nabla\chi|)^2.
  \end{split}\]
\end{proof}


\thm{SC} trivially holds if $\l(-L,\O)=-\infty$. Hence, the case 1 is a 
consequence of the following result.

\begin{proposition}
Under the assumptions of \thm{SC} case 1, any $\lambda<\l(-L,\O)$ admits
a (positive) eigenfunction in $\O$ with positive exponential
growth.
\end{proposition}

\begin{proof}
{\em Case $N=1$ and $\O=\R$}.\\
For $n\in\N$, let $v_n$ be the solution of
$(L+\lambda)v_n=0$ in $(-n,n)$ satisfying $v_n(-n)=M>0$, $v_n(n)=0$, with $M>0$
such that $v_n(0)=1$.
Note that $\l(-L,(-n,n))>\lambda$ and then $v_n$ is positive by the \MP.
By elliptic estimates and Harnack's inequality, $(v_n)_{n\in\N}$
converges (up to subsequences) locally uniformly to a nonnegative solution $v$
of $(L+\lambda)=0$ in $\R$. Since $v(0)=1$, the \SMP\ implies that $v$ is
positive.
We apply Lemma \ref{lem:estimate} to $v_n$, with $\chi=0$ in $(-\infty,0]$
and $\chi=1$ in $[1,+\infty)$. We derive
\[\begin{split}\int_1^n v_n^2 &\leq\int_0^n(v_n\chi)^2 \leq
\frac{\sup_\R\ol\alpha}{\l(-L,(-n,n))-\lambda}\int_0^n(v_n\chi')^2\\
&\leq
\frac{\sup_\R\ol\alpha}{\l(-L,\R)-\lambda}\int_0^1(v_n\chi')^2.
\end{split}\]
Whence, letting $n\to\infty$, $\int_1^{+\infty}v^2<+\infty$.
Next, consider a family $(\chi_n)_{n\in\N}$ of smooth functions satisfying
$$\chi_n(x)=0\quad\text{for }|x|\geq n,\qquad
\chi_n(x)=1\quad\text{for }|x|\leq n-1,\qquad
|\chi_n'|\leq2\quad\text{in }\R.$$
Lemma \ref{lem:estimate} yields
$$\int_{n-1\leq|x|\leq n}v^2\geq\frac14\int_{\R}(v\chi_n')^2
\geq\frac{\l(-L,\R)-\lambda}{\ol\alpha}\int_{|x|\leq n-1}v^2.$$
There exist then $k,\e>0$ such
that
$$\forall n\in\N,\quad \int_{n-1\leq|x|\leq n}v^2\geq
k(1+\e)^n.$$
Since $\int_0^{+\infty}v^2<+\infty$, it follows that $\int_{-n\leq
x\leq-n+1}v^2\geq \frac k2(1+\e)^n$, for $n$ large enough.
Hence, we can find $-n<x_n<-n+1$ such that $v(x_n)\geq\sqrt k(1+\e)^{\frac n2}$.
By Harnack's inequality we deduce that $v$ has positive exponential growth at
$-\infty$. 
Changing $x$ in $-x$ in the coefficients of $L$ and applying the above arguments
yields the existence of a positive solution $w$ of $L+\lambda=0$ in $\R$ 
which has positive exponential growth at $+\infty$. The function $v+w$
is an eigenfunction associated with $\lambda$ with positive exponential
growth.

{\em Case $N=1$ and $\O$ is a half-line}.\\
We can assume, without loss of generality, that $\O=(0,+\infty)$.
Let $\lambda<\l(-L,\O)$ and let $u$ be a positive solution of $(L+\lambda)=0$
in $\O$ satisfying $u(0)=0$.
For $n\in\N$, applying Lemma \ref{lem:estimate} with $\chi=1$ in $[0,n-1]$,
$\chi=0$ in $[n,+\infty)$ and $|\chi'|\leq2$ in $[n-1,n]$, we obtain
$$\int_0^{n-1} u^2\leq\e\int_{n-1}^n u^2,$$
for some $\e$ independent of $n$. The same argument as above shows that $u$ has
exponential growth at $+\infty$.

{\em Case $N>1$ and $L$ is radially symmetric}.\\
For $\lambda<\l(-L,\R^N)$, there exists a positive solution $u$ of
$(L+\lambda)=0$ in $\R^N$, which in addition is radially symmetric.
Applying Lemma \ref{lem:estimate} with $\chi=1$ in $B_{n-1}$,
$\chi=0$ outside $B_n$ and $|\nabla\chi|\leq2$ in $B_n\backslash B_{n-1}$, we
get
$$\int_{B_{n-1}}u^2\leq\e\int_{B_n\backslash B_{n-1}}u^2,$$
for some $\e$ independent of $n$. As a consequence, $\int_{B_n\backslash
B_{n-1}}u^2$ grows exponentially in $n$ and then there exists a sequence
$(x_n)_{n\in\N}$ with the same property and $n-1<|x_n|<n$. The symmetry of $u$
together with Harnack's inequality imply that $u$ has exponential growth.
%
\end{proof}


\subsection{Continuity of $\l$ with respect to decreasing sequences of domains}
\label{sec:domain}

We know that $\l$ is continuous with respect to increasing sequences of
domains (see statement (iv) of Proposition \ref{l1}).
We now derive the continuity property for sequences of sets approaching the
domain from outside - \thm{l1dec}.

Let us first sketch how one can derive the property in the case $\O$ bounded
and smooth. Owing to the monotonicity of $\l$ with respect to the inclusion of
domains, it is sufficient to prove the result in the case
$\O_n=\bigcup_{x\in\O}B_{1/n}(x)$. To prove that
$\lambda^*:=\limn\l(-L,\O_n)=\l(-L,\O)$, one considers the Dirichlet principal
eigenfunction $\vp_1^n$ of $-L$ in $\O_n$, normalized by
$\|\vp_1^n\|_{L^\infty(\O_n)}=1$. By elliptic estimates, $(\vp_1^n)_{n\in\N}$
converges (up to subsequences) to a solution $\t\vp_1$ of
$-L\t\vp_1=\lambda^*\t\vp_1$ in $\O$. Moreover, since the $\O_n$ are uniformly
smooth for $n$ large, the
$C^1$ estimates up to the boundary yield $\t\vp_1=0$ on $\partial\O$
and $\|\t\vp_1\|_{L^\infty(\O)}=1$. Hence, $\t\vp_1$ is the
Dirichlet principal eigenfunction of $-L$ in $\O$, that is,
$\lambda^*=\l(-L,\O)$.

Three types of difficulties arise in the general case. First, 
if $\O$ is not smooth one has to consider generalized principal eigenfunctions
satisfying the Dirichlet boundary conditions in the relaxed sense of \cite{BNV}.
In particular, the $C^1$ boundary estimates are no
longer available and then the passage to the limit in
the boundary conditions is a subtle issue. Second, if $\O$ is unbounded
then it might happen that $\l(-L,\O_n)=-\infty$ for all $n\in\N$. Third,
in unbounded domains, the
existence of a (positive) eigenfunction vanishing on the boundary does not
characterize $\l$, as shown by \thm{spettro}. 

\begin{proof}[Proof of \thm{l1dec}]
The proof is divided into three steps.

{\em Step 1:\, reducing to domains with smooth boundary portions.}
\\
We want to replace $(\O_n)_{n\in\N}$
with a family of domains
$(\mc{O}_r)$ having uniformly smooth boundaries in a neighborhood
of $\ol{\O_1\backslash\O}$.
Let $U$ be a bounded neighborhood of $\ol{\O_1\backslash\O}$ such that
$\ol U\cap\partial\O$ is smooth.
Consider a nonnegative smooth function $\chi$ defined on $\ol U\cap\partial\O$
which is positive on $\ol{\O_1\backslash\O}\cap\partial\O$ and whose support is
contained in $U$. Then, for
$r>0$, define
$$\mc{O}_r:=\O\cup\{\xi+\delta\nu(\xi)\ :\ \xi\in U\cap\partial\O,\ 
0\leq\delta<\frac1r\chi(\xi)\}.$$
where $\nu(\xi)$ stands for the outer normal to $\O$ at $\xi$.
The smoothness of $U\cap \partial\O$ implies the existence of $r_0>0$
such that the $(\partial{\mc O}_r)_{r\geq r_0}$
are uniformly smooth in $U$.
It is left to the reader to show that, for all $n\in\N$, there exists $k_n\in\N$
such that $\O_{k_n}\subset\mc{O}_n$.
Hence, by Proposition \ref{l1}, the sequences
$(\l(-L,\O_n))_{n\in\N}$, $(\l(-L,\mc{O}_n))_{n\in\N}$ are nondecreasing
and satisfy
$$\lambda^*:=\limn\l(-L,\mc{O}_n)\leq
\limn\l(-L,\O_n)\leq\l(-L,\O).$$ 
To prove the result it is then sufficient to show that $\lambda^*\geq\l(-L,\O)$.
We argue by contradiction assuming that
$\lambda^*<\l(-L,\O)$. 

{\em Step 2:\, the case $\O$ bounded.}
\\
Let $\vp_1^n$ be the
generalized principal eigenfunction of $-L$ in
$\mc{O}_n$, provided by Property A.1, normalized by
$\|\vp_1^n\|_{L^\infty(\mc{O}_n)}=1$. For given
$\t\lambda\in(\lambda^*,\l(-L,\O))$, we have that $\l(-(L+\t\lambda),\O)>0$ and
that $(L+\t\lambda)\vp_1^n>0$ a.e.~in $\O$. Thus, the refined
Alexandrov-Bakelman-Pucci estimate -
Property A.6 - yields
$$\sup_{\O}\vp_1^n\leq\sup_{\mc{O}_n\cap\partial\O}\vp_1^n\left(1+
A(\sup_{\O}c^++\t\lambda)|\O|^{1/N}\right),$$
for some $A$ independent of $n$. On the other hand, for $n$ large
enough, every
$x\in\mc{O}_n\backslash\O$ satisfies
$\dist(x,\partial\mc{O}_n)\leq\frac1n\sup\chi$. Therefore, since $\vp_1^n$
vanishes on $\partial\mc{O}_n$, the local boundary estimate
and the uniform smoothness of $(\partial\mc{O}_n)_{n\geq r_0}$ in $U$ yield
$$\limn\sup_{\mc{O}_n\backslash\O }\vp_1^n=0.$$
We eventually get $\limn\|\vp_1^n\|_{L^\infty(\mc{O}_n)}=0$, which 
is a contradiction.

{\em Step 3:\, the general case.}
\\
Suppose that $0\in\O$. For $\rho>0$, let $\mc{B}_\rho$ denote the connected
component of $\O\cap B_\rho$ containing the origin.
Since $\ol U\cap\partial\O$ is smooth, a compactness argument shows that there
exists $\rho_0>0$ such that $(\O\cap U)\subset\mc{B}_{\rho_0}$.
Thus, for all $\rho\geq\rho_0$ and $n\in\N$,
the set $(\mc{O}_n\backslash\O)\cup\mc{B}_\rho$ is connected.
Fix $\t\lambda\in(\lambda^*,\l(-L,\O))$.
Proposition \ref{l1} part (iv) yields
$$\fa
n\in\N,\quad\lim_{\rho\to\infty}\l(-L,(\mc{O}_n\backslash\O)\cup\mc{B}_\rho)=
\l(-L,\mc{O}_n)\leq\lambda^*<\t\lambda.$$
It is then possible to choose $\rho_n>\rho_0$ in such a way that
$\l(-L,(\mc{O}_n\backslash\O)\cup\mc{B}_{\rho_n})<\t\lambda$. We can assume,
without loss of generality, that the sequence $\seq{\rho}$ is increasing and
diverging. For fixed $n\in\N$, $n\geq r_0$, consider the following mapping:
$$\Theta(r):=\l(-L,(\mc{O}_r\backslash\O)\cup\mc{B}_{\rho_n}).$$
We know that $\Theta(n)<\t\lambda$ and, by step 2 and Proposition
\ref{l1} part (iv), that $\Theta$ is continuous on $[n,+\infty)$
and satisfies
$$\lim_{r\to+\infty}\Theta(r)=\l(-L,\mc{B}_{\rho_n})>\l(-L,\O)>\t\lambda.$$
Hence, there exists $r_n>n$ such that
$\Theta(r_n)=\t\lambda$.
We set $\t{\mc{O}}_n:=(\mc{O}_{r_n}\backslash\O)\cup\mc{B}_{\rho_n}$.
By \thm{spettro} and Remark \ref{rem:portion}, there exists a function
$\vp_1\in W^{2,p}_{loc}(\O\cup(U\cap\partial\O))$, $\fa p<\infty$, satisfying
$$\vp_1>0\quad\text{in }\O,\qquad
-L\vp_1=\l(-L,\O)\vp_1\quad\text{a.e.~in
}\O,\qquad
\vp_1=0\quad\text{on }U\cap\partial\O.$$
Let $\t\vp_1^n$ be the generalized principal eigenfunction
in $\t{\mc{O}}_n$, vanishing on $U\cap\partial\t{\mc{O}}_n$, given by Property
A.1. We now use the following

\begin{lemma}\label{lem:Hopf}
Let $\mc{O}$ and $K$ be an open and a compact subset of $\R^N$
such that $T:=\partial\mc{O}\cap(K+B_\e)$ is smooth for some
$\e>0$, and let $v\in W^{2,N}_{loc}(\mc{O})$ be positive and satisfy
$Lv\leq0$ a.e.~in $\mc{O}$.
Then, there exists a positive constant $h$ such that
$$\sup_{\mc{O}\cap K}\frac u v\leq h\|u\|_{W^{1,\infty}(\mc{O}\cap(K+B_\e))},$$
for all $u\in C^1(\mc{O}\cup T)$ satisfying $u\leq0$ on $T$.
\end{lemma}

Let us postpone the proof of Lemma \ref{lem:Hopf} and continue
with the one of \thm{l1dec}.
Consider a neighborhood $V$ of $\supp\chi$ such that
$\ol V\subset U$.
Applying Lemma \ref{lem:Hopf} with $\mc{O}=\O\cap U$, $K=\ol\O\cap\partial
V$, $u=\t\vp_1^n$ and $v=\vp_1$, we see that it is possible to normalize the
$\t\vp_1^n$ in such a way
that
\Fi{zn} \fa n\geq r_0,\quad\inf_{\O\cap\partial
V}\frac{\vp_1}{\t\vp_1^n}=1. \Ff
Note that the generalized
\pe\ $\l$ of $-(L+\t\lambda)$ is positive in any connected component of
$\O\cap\t{\mc{O}}_n\backslash\ol V$. Hence, owing to Property A.2,
$\vp_1\geq\t\vp_1^n$ in $\O\cap\t{\mc{O}}_n\backslash\ol V$ by the
refined MP.
Moreover, since the $(\partial\t{\mc{O}}_n)_{n\geq
r_0}$ are uniformly smooth in $U$, using the boundary Harnack inequality and
the local boundary estimate we infer that, for any compact $K\subset \O\cup
(U\cap\partial\O)$, the
$\t\vp_1^n$ are uniformly bounded in $W^{2,p}(K)$ for $n$ large enough. Thus, by
Morrey's inequality, they converge (up
to subsequences) in $C^1_{loc}(\O\cup(U\cap\partial\O))$ to a nonnegative
function
$\t\vp_1\in W^{2,p}_{loc}(\O\cup(U\cap\partial\O))$ satisfying
\Fi{tilde}
\left\{\begin{array}{ll}
-L\t\vp_1=\t\lambda\t\vp_1 & \text{a.e.~in }\O\\
\t\vp_1=0 & \text{on }U\cap\partial\O.\\
\end{array}\right.
\Ff
Furthermore, $\t\vp_1\leq\vp_1$ in $\O\backslash V$ and then in the whole $\O$
by the refined MP.
Therefore, the difference $\vp_1-\t\vp_1$ is a nonnegative strict supersolution
of
\eq{tilde}. The strong \MP\ implies $\vp_1-\t\vp_1>0$ in
$\O$.
Applying Lemma \ref{lem:Hopf} with $u=\t\vp_1$, $v=\vp_1-\t\vp_1$ and
$L=L+\t\lambda$, we can find a positive constant $h$ such that
$\t\vp_1\leq h(\vp_1-\t\vp_1)$ in $\ol\O\cap\partial V$, i.e.,
$\vp_1\geq(1+h^{-1})\t\vp_1$.
Since $\t\vp_1^n$ converges to $\t\vp_1$ in
$C^1_{loc}(\O\cup(U\cap\partial\O))$,
using again Lemma \ref{lem:Hopf} we can choose $n$
large enough in such a way that $(2h+2)^{-1}\vp_1\geq\t\vp_1^n-\t\vp_1$
in $\O\cap\partial V$. Gathering together
these inequalities we derive
$$\vp_1\geq(1+h^{-1})(\t\vp_1^n-(2h+2)^{-1}\vp_1)=
(1+h^{-1})\t\vp_1^n-(2h)^{-1}\vp_1\quad\text{in }\O\cap\partial V.$$
This contradicts \eq{zn}.
\end{proof}

It remains to prove Lemma \ref{lem:Hopf}. It is
essentially a consequence of the Hopf lemma, even though the hypothesis on $v$
does not allow one to apply it in its classical form.

\begin{proof}[Proof of Lemma \ref{lem:Hopf}]
Assume by contradiction that there exist a sequence of functions
$(u^n)_{n\in\N}$ with $u^n\leq0$
on $T$, $\|u^n\|_{W^{1,\infty}(\mc{O}\cap(K+B_\e))}=1$ and a
sequence of points $\seq{x}$ in $\mc{O}\cap K$ such that
$u^n(x_n)>nv(x_n)$. Let $\xi$ be the limit of (a subsequence of)
$\seq{x}$. It follows that $\xi\in \partial\mc{O}\cap K$. Let $\xi_n$ be one of
the projections of $x_n$ on $\partial\mc{O}$. Clearly, $x_n,\xi_n\in K+B_\e$ for
$n$ large enough. Thus,
\Fi{ri}
\limsup_{n\to\infty}\frac{v(x_n)}{|x_n-\xi_n|}\leq\limsup_{n\to\infty}\frac{
u^n(x_n)-u^n(\xi_n)}{n|x_n-\xi_n|}=0.
\Ff
On the other hand, since $T$ is smooth, there exists $R>0$
such that $\mc{O}$ satisfies the interior sphere condition of radius
$R$ at the points $\xi_n$, for $n$ large enough. That is, $x_n\in
B_R(y_n)\subset\mc{O}$,
where $y_n:=\xi_n-R\nu(\xi_n)$ and  Fix $\rho\in(0,R)$. The existence of the
positive supersolution $v$,
together with Proposition \ref{l1} part (iii), imply that
$0\leq\l(-L,\mc{O})<\l(-L,B_R(y_n)\backslash B_\rho(y_n))$.
Therefore, owing to Property A.2 in Appendix \ref{sec:BNV},
one can follow the standard argument used to prove the Hopf lemma (see, e.g.,
\cite{Max} or Lemma 3.4 in \cite{GT}), comparing $v$ with an exponential
subsolution, and find a positive constant $\kappa$
such that, for $n$ large enough,
$$\fa x\in B_R(y_n)\backslash B_\rho(y_n),\quad
\frac{v(x)}{R-|x-y_n|}\geq\kappa\min_{\partial B_\rho(y_n)}v.$$
This contradicts \eq{ri}.
\end{proof}

\thm{l1dec} does not hold in general if
$\O_1\backslash\O$ is not bounded, as shown by Example 1.9 in \cite{Sznitman}.
The smoothness hypothesis on $\partial\O$ is also necessary, because it is
possible to
find two bounded domains $\O\subset\O'$ satisfying $\ol\O=\ol{\O'}$ and 
$\l(-\Delta,\O)>\l(-\Delta,\O')$ (see Remark
\ref{rem:strict} below). Hence, the sequence $\seq{\O}$
identically
equal to $\O'$ violates the convergence result.


\begin{remark}\label{rem:strict}
 If $\O\subset\O'$ are bounded and there exist $\xi\in\O'\cap\partial\O$ and
$\delta>0$ such that $\O\cap B_\delta(\xi)$ has a connected component $U$
satisfying the exterior cone condition at $\xi$ (or, more generally, admitting a
strong barrier at $\xi$) then $\l(-L,\O)>\l(-L,\O')$. To see this, consider the
generalized principal eigenfunctions $\vp_1$ and $\vp'_1$ of $-L$ in $\O$ and
$\O'$ respectively, given by Property A.1.
The function $\vp_1$ can be obtained as the limit of the classical
Dirichlet principal eigenfunctions of $-L$ in a family of smooth domains
invading $\O$, normalized by $\|\.\|_\infty=1$. As a consequence, the existence
of the barrier function at $\xi$ yields
$\lim_{\su{x\in U}{x\to\xi}}\vp_1(x)=0$. Since $\vp_1'(\xi)>0$ by the strong
\MP, we infer that $\vp_1$ and $\vp_1'$ are linearly independent. Therefore, 
Property A.4 implies that $\l(-L,\O)>\l(-L,\O')$.
 
Note that $\O$, $\O'$ fulfill the above property as soon as
$\O'\backslash\O$ contains a $N-1$-dimensional Lipschitz manifold.
\end{remark}

\begin{remark}
If $\O$ is bounded then the arguments in the proof of \thm{l1dec} work, with
minor modifications, only assuming that $\partial\O$ is Lipschitz
in a neighborhood of $\ol{\O_1\backslash\O}$. We do not know if the result
holds for unbounded Lipschitz domains.
\end{remark}

The first step of the proof of \thm{l1dec} consists in showing that
the $\O_n$
approach $\O$ in the sense of the Hausdorff distance $d_H$
\footnote{For $A, B\subset\R^N$, $d_H(A,B):=\max\left(\rho(A,B),\rho(B,A)
\right)$,
where
$\displaystyle\rho(A,B):=\sup_{a\in A}\inf_{b\in B}|a-b|$.}. Remark
\ref{rem:strict} shows that, in the non-smooth case, $\l(-L,\.)$ is not
continuous with respect to $d_H$, and this is why the result of \thm{l1dec} may
fail in that case. Note, however, that the domains $\O$, $\O'$ in Remark
\ref{rem:strict}
satisfy $d_H(\O^c,(\O')^c)>0$.
The Hausdorff distance between the complements is a better suited notion of
distance for open sets (it implies for instance that if
$d_H(\O_n^c,\O^c)\to0$ then
$\,\inter(\bigcap_{n\in\N}\O_n)\subset\O\subset\bigcup_{n\in\N}\O_n$).
A consequence of a $\gamma$-convergence result by \v{S}ver{\'a}k
\cite{Sverak} is that if $N=2$, $L$ is self-adjoint and $(\O_n)_{n\in\N}$ is a
sequence of uniformly bounded domains, such that the number of connected
components of $\O_n^c$ is uniformly bounded and $\limn
d_H(\O_n^c,\O^c)=0$, then $\limn\l(-L,\O_n)=\l(-L,\O)$. We refer to \S 2.3.3
in \cite{Henrot} for other
continuity results for self-adjoint operators in bounded domains obtained via
$\gamma$-convergence.
Always in the case of bounded domains, A.-S.
Sznitman proves in \cite{Sznitman}, Proposition 1.10, using a probabilistic
approach, that the continuity of
$\l$ with respect to decreasing sequences of domains $(\O_n)_{n\in\N}$ holds
without any smoothness hypothesis on $\partial\O$, provided that
$\bigcap_n\O_n=\O$. This hypothesis, which is stronger than
$\limn d_H(\O_n^c,\O^c)=0$, is quite
restrictive because, in general, $\bigcap_n\O_n$ is not an open set.

\subsection{Proof of \thm{SC}, cases 2-4}

Below, we give a characterization of $\l''$ which provides a
necessary and sufficient condition for the equivalence between $\l$, $\l'$ and
$\l''$. This characterization emphasizes that $\l''$ strongly
reflects the properties of the operator at both finite distance and 
infinity.

\begin{theorem}\label{thm:equivalence}
If $\O$ is unbounded and smooth then
$$\l''(-L,\O)=\min\left(\l(-L,\O),\lim_{r\to\infty}\l''(-L,\O\backslash\ol{B_r}
)\right).$$
As a consequence, $\l(-L,\O)=\l''(-L,\O)$ ($=\l'(-L,\O)$ if \eq{ABC3} holds)
iff
\Fi{l1<l1''r}
\lim_{r\to\infty}\l''(-L,\O\backslash\ol{B_r})\geq\l(-L,\O).
\Ff
\end{theorem}

\begin{proof}
We first note that definitions \eq{l1} and \eq{l1''} make good sense even if
$\O$ is not connected, and that statements (ii), (iii) of Proposition
\ref{pro:l1''} still hold in this case.
Thus, the function $\lambda''(r):=\l''(-L,\O\backslash\ol{B_r})$ is
nondecreasing with respect to $r$ and satisfies
$$\l''(-L,\O)\leq\lim_{r\to\infty}\lambda''(r)\leq+\infty.$$
Hence, since $\l''(-L,\O)\leq\l(-L,\O)$ by definition, we find that
$$\l''(-L,\O)\leq\min\left(\l(-L,\O),\lim_{r\to\infty}\lambda''(r)\right).$$
To prove the reverse inequality, let us show that
if there exists $\lambda\in\R$ satisfying
$$\lambda<\min\left(\l(-L,\O),\lim_{r\to\infty}\lambda''(r)\right),$$
then $\l''(-L,\O)\geq\lambda$.
Take $R>0$ such that $\lambda''(R)>\lambda$.
We first prove the result in the case $\O=\R^N$. The proof in the general
case is more involved and makes use of an auxiliary result - Lemma
\ref{lem:infT>0} below - derived from \thm{l1dec}.

Since $\lambda''(R)>\lambda$, there exists $\phi\in W^{2,N}_{loc}(\R^N\backslash
\ol B_R)$
with positive infimum and such that $(L+\lambda)\phi\leq0$ a.e.~in $\R^N\backslash
\ol B_R$.
By Proposition \ref{pro:C1} and Morrey's inequality, we can assume without loss of
generality that $\phi\in C^1(B_{R+1}^c)$, where $B_{R+1}^c=\R^N\backslash
B_{R+1}$.
Let $\vp$ be an eigenfunction associated with $\l(-L,\R^N)$
(provided by statement (v) of Proposition \ref{l1}) and $\chi\in C^2(\R^N)$ be
nonnegative and satisfy
$$\chi=0\quad\text{in }B_{R+1},\qquad\chi=1\quad\text{outside }B_{R+2}.$$
For $\e>0$, define the function $u:=\vp+\e\chi\phi$.
We see that $(L+\lambda)u\leq0$ a.e.~in $B_{R+1}\cup B_{R+2}^c$.
On the other hand, for a.e.~$x\in B_{R+2}
\backslash B_{R+1}$,
\[\begin{split}   
(L+\lambda)u &\leq(L+\lambda)\vp+
\e[\chi(L+\lambda)\phi+2a_{ij}\partial_i\chi\partial_j\phi
+(a_{ij}\partial_{ij}\chi+b_i\partial_i\chi)\phi]\\
&\leq(\lambda-\l(-L,\R^N))\vp+\e C,
\end{split}\]
where $C$ is a constant depending on $N$, the $L^\infty$ norms of $a_{ij}$,
$b_i$, the $W^{2,\infty}$ norm of $\chi$ and the $W^{1,\infty}$ norm of
$\phi$ on $B_{R+2}\backslash B_{R+1}$. Therefore, for $\e$ small enough the
function $u$ satisfies $(L+\lambda)u<0$ a.e.~in $B_{R+2}
\backslash B_{R+1}$. Since $u$ is an admissible function for $\l''$, we
eventually obtain $\l''(-L,\R^N)\geq\lambda$.

Let us now turn to the case of a general smooth domain $\O$.
Assume that $\O\cap B_R\neq\emptyset$, otherwise we immediately get
$\l''(-L,\O)=\lambda''(R)>\lambda$.
The open set $\O\backslash \ol B_R$, being smooth in a neighborhood
of $\partial B_{R+1}$, has a finite number of connected components
$\mc{O}_1,\dots,\mc{O}_m$ intersecting $\partial B_{R+1}$. This is seen by a
compactness argument that we leave to the reader.
For $j\in\{1,\dots m\}$, we have
$\l''(-L,\mc{O}_j)\geq\l''(R)>\lambda$.
Since $\partial\mc{O}_j\backslash\partial B_R$ is smooth, by Proposition \ref{pro:C1} 
there exists a function
$\phi^j\in W^{2,p}_\loc(\ol{\mc{O}}_j\backslash\partial B_R)$, $\fa p<\infty$,
satisfying
\Fi{Oj}
\inf_{\mc{O}_j}\phi^j>0,\qquad (L+\lambda)\phi^j\leq0\quad
\text{a.e.~in }\mc{O}_j.
\Ff
Define the function $\phi$ by setting $\phi(x):=\phi^j(x)$ if $x\in \mc{O}_j$.
Note that $\O\backslash B_{R+1}\subset\bigcup_{j=1}^m\mc{O}_j$ because
$\O$ is connected. Thus, $\phi\in W^{2,p}_\loc(\ol\O\backslash B_{R+1})$
satisfies \eq{Oj} with $\mc{O}_j$ replaced by
$\O\backslash \ol B_{R+1}$. We fix $\t\lambda\in(\lambda,\l(-L,\O))$ and
consider a function $\t\vp$ satisfying 
$$-L\t\vp=\t\lambda\t\vp\quad\text{a.e.~in }\O,\qquad
\t\vp>0\quad\text{in }\O\cup(\ol B_{R+2}\cap\partial\O).$$
The function $\t\vp$ replaces the
eigenfunction $\vp$ used in the case $\O=\R^N$. Its existence is given by the
next lemma.

\begin{lemma}\label{lem:infT>0}
Assume that $\O$ has a $C^{1,1}$ boundary portion $T\subset\partial\O$ which is
compact.
Then, for any $\t\lambda<\l(-L,\O)$, there exists $\t\vp\in
W^{2,p}_\loc(\O\cup T)$, $\fa p<\infty$, such that
$$-L\t\vp=\t\lambda\t\vp\quad\text{a.e.~in }\O,\qquad
\t\vp>0\quad\text{in }\O\cup T.$$
\end{lemma}

Postponing the proof of Lemma \ref{lem:infT>0} for a moment, let us complete
the proof of \thm{equivalence}.
Consider the same function $\chi\in C^2(\R^N)$ as before.
For $\e>0$, the function $u:=\t\vp+\e\chi\phi$ satisfies
$(L+\lambda)u\leq0$ a.e.~in
$\O\cap(B_{R+1}\cup B_{R+2}^c)$.
Moreover, since $\phi\in C^1(\ol\O\cap(\ol B_{R+2}\backslash B_{R+1}))$, the
same computation as before shows that there exists $C$ independent of $\e$ such
that
$$(L+\lambda)u\leq(\lambda-\l(-L,\t\O))\t\vp+\e C\quad\text{a.e.~in }
\O\cap(B_{R+2}\backslash B_{R+1}).$$
The latter quantity is negative for $\e$ small enough because $\t\vp>0$ on
$\ol{\O\cap B_{R+2}}$. Therefore, taking $\phi=u$ in \eq{l1''}
we get $\l''(-L,\O)\geq\lambda$. 

The last statement of \thm{equivalence} follows immediately from
\thm{relations}.
\end{proof}

\begin{proof}[Proof of Lemma \ref{lem:infT>0}]
Let $U$ be a bounded neighborhood of $T$ where $\partial\O$ is smooth.
Consider an extension of the operator $L$ - still denoted by $L$ - to $\O\cup
U$, satisfying the same hypotheses as $L$.
As we have seen in the proof of \thm{l1dec}, it is possible to construct a
decreasing sequence of domains $(\mc{O}_n)_{n\in\N}$ satisfying
$$\ol{\mc{O}_1\backslash\O}\subset U,\qquad
\fa n\in\N,\quad\O\cup T\subset\mc{O}_n,\qquad
\bigcap_{n\in\N}\ol{\mc{O}}_n=\ol\O.$$
Hence, by \thm{l1dec}, $\l(-L,\mc{O}_n)>\t\lambda$ for $n$ large enough.
It then follows that there exists a positive
function $\t\vp\in W^{2,p}_\loc(\mc{O}_n)$, $\fa p<\infty$, satisfying
$-L\t\vp=\t\lambda\t\vp$ a.e.~in $\mc{O}_n$. In particular, $\t\vp>0$ on
$\O\cup T\subset \mc{O}_n$.
\end{proof}

\begin{proof}[Conclusion of the proof of \thm{SC}]
Cases 2-4 are derived from \thm{equivalence}, which is a powerful tool to
understand when equality occurs. Thus, the aim is to prove \eq{l1<l1''r}.

Case 2) By the definition of $\l''$, it follows that
\[\begin{split}
\lim_{r\to\infty}\l''(-L,\O\backslash\ol{B_r}) &\geq
\lim_{r\to\infty}\left(\l''(-\t L,\O\backslash\ol{B_r})-
\sup_{\O\backslash\ol{B_r}}\gamma\right)\\
& =\lim_{r\to\infty}\l''(-\t L,\O\backslash\ol{B_r}).
\end{split}\]
The last limit above is greater than or equal to $\l''(-\t L,\O)=\l(-\t
L,\O)$. Since $\gamma\geq0$, we see that $\l(-\t
L,\O)\geq\l(-L,\O)$. Hence, \eq{l1<l1''r} holds.

Case 3) Proposition \ref{pro:l1''} part (ii) yields
$$\lim_{r\to\infty}\l''(-L,\O\backslash\ol{B_r})\geq
\lim_{r\to\infty}(-\sup_{\O\backslash\ol{B_r}}c)=
-\limsup_{\su{x\in\O}{|x|\to\infty}}c(x)\geq\l(-L,\O).$$

Case 4)
Owing to the case 3, it is sufficient to show that $\l(-L,\O)\leq-\sigma$, for
all $\sigma<\limsup_{\su{x\in\O}{|x|\to\infty}}c(x)$. Take such a 
$\sigma$.
Consider first the case where $L$ is self-adjoint. Let $B$ be a ball
contained in $\O$. Proposition \ref{l1} part (iii) yields $\l(-L,\O)
\leq\l(-L,B)$. From the Rayleigh-Ritz formula, it then follows that
$$\l(-L,\O)\leq\l(-L,B)\leq\l(-\Delta,B)\sup_B\ol\alpha-\inf_B c.$$
Since, by hypothesis, we can find balls $B\subset\O$ with arbitrarily large
radius such that $\inf_B c>\sigma$, we deduce that
$\l(-L,\O)\leq-\sigma$.
Consider now the case where $L$ is not self-adjoint.
By hypothesis, there exists $\delta>0$ such that, for all $r>0$, there is a
ball $B$ of radius $r$ satisfying
$$\fa x\in B,\quad 4\ul\alpha(x)(c(x)-\sigma)\geq\delta.$$
Let $B'$ be another ball of radius $r/4$ contained in the set $B\backslash
B_{r/2}$.  For large enough $r$, we find that
$$\fa x\in B',\quad 4\ul\alpha(x)(c(x)-\sigma)-|b(x)|^2\geq\delta/2.$$
As shown in Lemma 3.1 of \cite{BHRossi}, if the radius of $B'$ is large enough
(depending on $\delta$), the above condition ensures the existence of a $C^2$
function $\phi$ satisfying
$$(L-\sigma)\phi>0\text{ in }B',\qquad
\phi>0\text{ in }B',\qquad\phi=0\text{ in }\partial B'.$$
As a consequence,
$$-\sigma\geq\l'(-L,B')=\lambda_{B'}=\l(-L,B')\geq\l(-L,\O).$$
\end{proof}

\begin{remark}
If the function $\gamma$ in the case 2 of \thm{SC} is
compactly supported in $\O$, then $\l(-L,\O)=\l''(-L,\O)$ holds
true even for $\O$ non-smooth.
\end{remark}


\section{Existence and uniqueness of the principal eigenfunctions}
\label{sec:simplicity}

We now investigate the simplicity of $\l$. Another natural question
is to know whether the generalized \pe s $\l'$, $\l''$ have corresponding
eigenvalues that satisfy the additional requirements of their definitions. This
section is devoted to these questions.

We say that an eigenfunction $\vp$ is admissible for $\l'$ (resp.~$\l''$) if it satisfies
$$\sup_\O\vp<\infty\qquad\text{(resp. }\fa\e>0,\quad
\inf_{\O_\e}\vp>0),$$
where $\O_\e$ is defined in Section \ref{sec:other}.
Throughout this section, we assume that $\l,\l',\l''\in\R$ (which is for
instance the case if $\sup c<+\infty$).

From \thm{spettro} we know that if $\O$ is smooth then there always exist
eigenfunctions with eigenvalues $\l$, $\l'$, $\l''$ respectively.
But, as we show below, $\l'$ and $\l''$ may not have admissible eigenfunctions.
Moreover, $\l$, $\l'$, $\l''$ are generally not simple.

\begin{proposition}\label{pro:nopef}
There exist operators $L$ for which there are several linearly independent 
eigenfunctions associated with the eigenvalues $\l(-L,\O)$, $\l'(-L,\O)$, $\l''(-L,\O)$.
There are also operators such that
$\l'$ or $\l''$ have several linearly independent admissible eigenfunctions and others
for which they do not have any.
\end{proposition}

\begin{proof}
Let
$Lu=u''+c(x)u$ in $\R$, with $c<0$ in $(-1,1)$ and $c=0$ outside. We show that
$\l'$ has no admissible eigenfunctions and that
$\l(-L,\R)=\l'(-L,\R)=\l''(-L,\R)=0$ is not simple, even in the class of
admissible eigenfunctions for $\l''$. 
Let $\vp_-$ and $\vp_+$ be the solutions to $Lu=0$ in $\R$
satisfying $\vp_{\pm}(\pm1)=1,\ \vp_{\pm}'(\pm1)=0$. By ODE
arguments we find that $\vp_-$ and $\vp_+$ are positive and satisfy
$$\vp_-=1\ \text{ in }(-\infty,-1],\qquad\vp_+=1\ \text{ in }[1,+\infty),\qquad
\lim_{x\to\mp\infty}\vp_\pm(x)=+\infty.$$
Consequently, they are linearly independent and thus they generate the space
of solutions to $Lu=0$ in $\R$. Taking $\phi=\vp_-$
in \eq{l1''} and using \thm{relations} we derive
$\l''(-L,\R)=\l'(-L,\R)=\l(-L,\R)=0$.

To exhibit an example of non-existence of admissible eigenfunctions for $\l''$,
we will make use of Proposition
\ref{pro:c<0}, proved at the end of this section.
Consider the operator $Lu=u''+c(x)u$ in $\R$, with $c=0$ in $(-\pi,\pi)$,
$c=-1$ outside $(-\pi,\pi)$. By Proposition \ref{l1} part (iii) we see that
$\l(-L,\R)<\l(-L,(-\pi,\pi))=1/4$. Thus, \thm{SC} yields
$\l''(-L,\R)=\l(-L,\R)$. But Proposition
\ref{pro:c<0} implies that the eigenfunction associated with $\l(-L,\R)$ is
unique (up to a scalar multiple) and vanishes at infinity.

Lastly, an example of non-uniqueness of admissible eigenfunctions for $\l'$ is
given by the operator
$$Lu:=u''+\frac{2x}{1+x^2}u'\ \text{ in }\R.$$
In fact, the functions $u_1\equiv1$ and $u_2(x)=\arctan(x)+\pi$ satisfy
$Lu=0$ in $\R$. Taking $\phi=u_1$ in the definition of $\l'$ and $\l''$
we get $\l'(-L,\R)\leq0\leq\l''(-L,\R)$. Hence,
$\l''(-L,\R)=\l'(-L,\R)=0$ by statement (iii) of \thm{relations} and, as a
consequence, $\l'$ is not simple.
\end{proof}

Let us mention two other examples of non non-existence of 
admissible eigenfunctions for $\l'$ and $\l''$ respectively, this time in 
higher dimension, that can be exhibited using the theory of critical operators
(see, e.g.,\cite{PiBook}). The first one is $L=\Delta + c (x)$ in $\R^2$, where
$$c(x)= \begin{cases} 1 &\text{if }x\in\bigcup_{n\in\N}B_{r_n}(x_n)\\
0 & \text{otherwise},
         \end{cases}$$
with $\seq{x}$, $\seq{r}$ such that the $B_{r_n}(x_n)$ are disjoint and
$|x_n|,r_n\to\infty$. Clearly, $\l(-L,\R^2)=\l'(-L,\R^2)=\l(-L,\R^2)=0$, but
one can show that the equation $L=0$ does not admit
positive bounded solutions in $\R^2$ (see \cite{Pinch07}).
An example where no admissible
eigenfunctions exist for $\l''$ is $L= \Delta + c (x)$ in $\R^N$, $N\geq3$, 
with $c\leq0$ chosen in such a way that $L$ is {\em critical}. 
Then $\l(-L,\R^N)= \l''(-L,\R^N)= 0$, and the (unique up to a scalar multiple) 
positive solution of $L=0$ behaves at infinity like $|x|^{2-N}$.

In order to derive
a sufficient condition for the simplicity of
$\l$, we introduce the notion of ``minimal
growth at infinity''. This notion slightly differs from the one
of S.~Agmon \cite{A1} (see
Remark \ref{rem:mgi} below). In the case of smooth domains, the
sufficient condition we obtain - \thm{mgi} - is more general than
Theorem 5.5 in \cite{A1} (whose proof can be found in \cite{Pinch90_2}, see
Lemma 4.6 and Remark 4.8 therein).
Let us mention that another sufficient condition for the
simplicity of $\l$ can be expressed in terms of the criticality
property of the operator (see, e.g., \S4 in \cite{PiBook} and the references
therein).

\begin{definition}\label{def:mgi}
Let $\O$ be unbounded.
A positive function $u\in W^{2,N}_{loc}(\O)$ satisfying
\Fi{L=0}
Lu=0\quad\text{a.e.~in }\O,
\Ff
is said to be a solution of \eq{L=0} of {\em minimal growth at
infinity} if for any $\rho>0$ and any positive function 
$v\in W^{2,N}_{loc}(\O\backslash B_\rho)$ 
satisfying $Lv\leq0$ a.e.~in $\O\backslash B_\rho$,
there exist $R\geq\rho$ and $k>0$ such that
$ku\leq v$ in $\O\backslash B_R$.
\end{definition}

\begin{remark}\label{rem:mgi}
Our Definition \ref{def:mgi} of minimal growth at infinity differs
from the original one of Agmon \cite{A1}. There, $B_\rho$ and $B_R$ are
replaced by two compact sets $K\subset K'\subset\O$. Thus, Agmon's definition regards
minimal growth both at infinity and at the boundary, whereas ours
only deals with behavior at infinity. Indeed,
Agmon calls it ``minimal growth at infinity in $\O$''.
Using the refined \MP\ in
bounded domains, one readily sees that solutions of minimal growth at infinity
vanishing on $\partial\O$ fulfill Agmon's definition. Hence, owing to Theorem
5.5
in \cite{A1}, they are unique up to a scalar multiple. This fact is expressed in
the next statement, whose simple proof is included here for the sake of 
completeness.
Another difference with Agmon's approach is that he also considers 
positive solutions in 
proper subsets $\O\backslash E$, without imposing condition on $\partial E$.
Such solutions can always be constructed, no matter 
what the sign of $\l(-L,\O)$ is, satisfying in addition the minimal growth 
condition. When $E$ reduces to a single point, this type of solutions is 
used to investigate the removability of singularities.
\end{remark}

\begin{proposition}[\cite{A1}]\label{pro:mgi}
Let $\O$ be unbounded and $u\in W^{2,N}_{loc}(\O)\cap
C^0(\ol\O)$ be a solution of \eq{L=0} of minimal growth at infinity
vanishing on $\partial\O$. Then, for any positive function $v\in
W^{2,N}_{loc}(\O)$ satisfying $Lv\leq0$
a.e.~in $\O$, there exists $\kappa>0$ such that
$v\equiv \kappa u$ in $\O$. In particular, $\l(-L,\O)=0$.
\end{proposition}

\begin{proof}
Taking $\phi=u$ in \eq{l1} yields $\l(-L,\O)\geq0$. Consider a function $v$ as
in the statement. The quantity
$$\kappa:=\inf_\O\frac v u$$
is a nonnegative real number.
Suppose by way of contradiction that $v-\kappa u>0$ in $\O$.
Applying Definition \ref{def:mgi} with $v-\kappa u$ in place of $v$, we can
find $R,h>0$ such that $hu\leq v-\kappa u$ in $\O\backslash B_R$.
By Property A.2, we know that the refined MP holds in
any connected component $\mc{O}$ of $\O\cap B_R$, because $\l(-L,\mc{O})>0$ by
Proposition \ref{l1} part (iii). As a consequence, $hu\leq v-\kappa u$ in the
whole $\O$. This contradicts the definition of $\kappa$. Therefore, $v-\kappa u$
vanishes somewhere in $\O$, and then everywhere by the strong \MP.
\end{proof}

From Proposition \ref{pro:mgi} it follows in particular that $\l$ is simple, in
the class of positive functions, as soon as it admits an eigenfunction 
having minimal growth at infinity (we recall that eigenfunctions are assumed to 
vanish on $\partial\O$). Here we derive a sufficient condition for 
this to hold.

\begin{theorem}\label{thm:mgi}
If $\O$ is unbounded and smooth and $\l(-L,\O)$ satisfies
$$\l(-L,\O)<\lim_{r\to\infty}\l(-L,\O\backslash\ol B_r),$$
then the associated eigenfunction is a solution of \eq{pef} of minimal growth at
infinity and, therefore, $\l(-L,\O)$ is simple in
the class of positive functions.
\end{theorem}

\begin{proof}
It is not restrictive to assume that $\l(-L,\O)=0$. Consider the
same family of bounded domains $\seq{\O}$ as in the proof of
\thm{spettro}, i.e.,
$$\fa n\in\N,\quad\O\cap B_n\subset\O_n\subset\O_{n+1}\subset\O.$$
As we have seen there, the generalized principal eigenfunctions
$\vp^n$ of $-L$ in $\O_n$ - provided by Property A.1 - normalized by 
$\vp^n(x_0)=1$, for a given $x_0\in\O$,
converge (up to subsequences) in $C^1_{loc}(\ol\O)$ to an
eigenfunction $\vp^*$ with eigenvalue $\l(-L,\O)$. We claim that $\vp^*$ is a
solution of
\eq{pef} of minimal growth at infinity.
By hypothesis, there exists $R>0$ such that
$\l(-L,\O\backslash\ol B_R)>0$.
Let $\mc{O}_1,\dots,\mc{O}_m$ be
the connected components of $\O\backslash\ol B_R$ intersecting
$\partial B_{R+1}$ (which are finite due to the smoothness of
$\O$). It follows from Proposition \ref{l1} that there is
$n_0\in\N$ such that
$$\fa j\in\{1,\dots m\},\ n\geq n_0,\quad\l(-L,\mc{O}_j)\geq
\l(-L,\O\backslash\ol B_R)>\l(-L,\O_n)>0.$$ Let $\phi^j>0$ satisfy
$-L\phi^j=\l(-L,\mc{O}_j)\phi^j$ a.e.~in $\mc{O}_j$ (see statement (v) of
Proposition \ref{l1}). 
Since $\vp^n\to\vp^*$ in $C^1_{loc}(\ol\O)$, by Lemma \ref{lem:Hopf}
it is possible to normalize $\phi^j$ in such a way
that
$$\fa n\in\N,\quad \phi^j\geq\vp^n\quad
\text{on }\O_n\cap\mc{O}_j\cap\partial B_{R+1}.$$
Hence, for $n\geq n_0$, applying the refined MP in every
connected component of $\O_n\cap\mc{O}_j\backslash\ol B_{R+1}$ -
which holds due to Property A.2 - we get
$\vp^n\leq\phi^j$ in $\O_n\cap\mc{O}_j\backslash B_{R+1}$.
It follows that, for given $\e>0$, the function $\vp^n-\e\phi^j$
satisfies
$$L(\vp^n-\e\phi^j)\geq[-\l(-L,\O_n)+\e\l(-L,\mc{O}_j)]\vp^n
\quad \text{a.e.~in }\O_n\cap\mc{O}_j\backslash B_{R+1}.$$
Therefore, since $(\l(-L,\O_n))_{n\in\N}$ converges to $0$,
there exists $n_1\in\N$ such that
$L(\vp^n-\e\phi^j)>0$ a.e.~in $\O_n\cap\mc{O}_j\backslash B_{R+1}$
for $n\geq n_1$. Consider now a function
$v$ as in Definition \ref{def:mgi}. Let $R'>\max(\rho,R+1)$.
By Lemma \ref{lem:Hopf}, there exists $h>0$ such that
$$\fa n\in\N,\quad h v\geq\vp^n\quad
\text{on }\O_n\cap\partial B_{R'}.$$
For $n\geq n_1$, applying once again the refined MP we then obtain
$\vp^n-\e\phi^j\leq hv$ in $\O_n\cap\mc{O}_j\backslash B_{R'}$.
Letting $n\to\infty$ we finally derive
$\vp^*-\e\phi^j\leq hv$ in $\mc{O}_j\backslash B_{R'}$.
Since the latter holds for all $j\in\{1,\dots m\}$ and $\e>0$,
we eventually infer that $\vp^*\leq hv$ in $\O\backslash B_{R'}$.
This concludes the proof.
\end{proof}

\begin{corollary}
If $\O$ is unbounded and smooth, the $a_{ij}$ are bounded and the
$b_i$, $c$ satisfy
\Fi{Binfty}
\lim_{\su{x\in\O}{|x|\to\infty}}\frac{b(x)\.x}{|x|}=\pm\infty,\qquad
\sup_\O c<\infty,
\Ff
then the eigenfunction associated with $\l(-L,\O)$ is a solution of \eq{pef}
of minimal growth at infinity, and it satisfies
$$\fa\sigma>0,\quad
\lim_{\su{x\in\O}{|x|\to\infty}}\vp(x)e^{\pm\sigma|x|}=0,$$
where the $\pm$ is in agreement with the $\pm$ in \eq{Binfty}.
\end{corollary}

\begin{proof}
For $\sigma>0$, define the function $\phi$ by
$\phi(x):=e^{\mp\sigma|x|}$, where the $\mp$
is in agreement with the $\pm$ in \eq{Binfty}.
The same computation as in the proof of Proposition
\ref{pro:l1inR} shows that $(L+\l(-L,\O)+1)\phi\leq0$
a.e.~in $\O\backslash B_r$, for $r$ large enough.
Therefore, $\l(-L,\O\backslash B_r)\geq\l(-L,\O)+1$.
The result then follows from \thm{mgi}.
\end{proof}

We now derive a result about the exponential decay of subsolutions
of the Dirichlet problem. This will be used to prove the last
statement of Proposition \ref{pro:c<0}.

\begin{proposition}\label{pro:decay}
Let $\O$ be unbounded and smooth, $L$ be an elliptic operator with
bounded
coefficients such that
$$\limsup_{\su{x\in\O}{|x|\to\infty}}c(x)<0,$$
and $A,\ B$ be the functions in \eq{AB}. Set
$$\Gamma_-:=\limsup_{\su{x\in\O}{|x|\to\infty}}\frac{B(x)-\sqrt{B^2(x)-4A(x)c(x)
}}
{2A(x)},$$
$$\Gamma_+:=\liminf_{\su{x\in\O}{|x|\to\infty}}\frac{B(x)+\sqrt{B^2(x)-4A(x)c(x)
}}
{2A(x)}.$$ Then, for any function $u\in W^{2,N}_\loc(\O)$ satisfying
$$Lu\geq0\text{ a.e.~in }\O,\qquad
\fa\xi\in\partial\O,\ \ \limsup_{x\to\xi}u(x)\leq0$$
and such that
$$\ex\gamma\in[0,-\Gamma_-),\quad
\limsup_{\su{x\in\O}{|x|\to\infty}}u(x)e^{-\gamma|x|}\leq0,$$
it holds that
$$\fa\eta\in(0,\Gamma_+),\quad
\limsup_{\su{x\in\O}{|x|\to\infty}}u(x)e^{\eta|x|}\leq0.$$
\end{proposition}

\begin{proof}
Let $\eta\in(0,\Gamma_+)$. Consider two numbers
$\ul\sigma\in(\Gamma_-,-\gamma)$ and
$\ol\sigma\in(\eta,\Gamma_+)$. By hypothesis, there exists $R>0$
such that, for a.e.~$x\in\O\backslash B_{R-1},\ u(x)\leq
e^{\gamma|x|},\ c(x)<0$ and
$$\ul\sigma>\frac{B(x)-\sqrt{B^2(x)-4A(x)c(x)}}
{2A(x)},\qquad
\ol\sigma<\frac{B(x)+\sqrt{B^2(x)-4A(x)c(x)}}{2A(x)}.$$ For any
$n\in\N$, define the function
$$\ol u_n(x):=e^{R(\gamma+\ol\sigma)-\ol\sigma|x|}+
e^{(R+n)(\gamma+\ul\sigma)-\ul\sigma|x|}.$$ Since for
$\sigma\in\R$ we have
$Le^{-\sigma|x|}=(A(x)\sigma^2-B(x)\sigma+c(x))e^{\sigma|x|}$, we
infer that $L\ol u_n\leq0$ a.e.~in $x\in\O\backslash\ol B_{R-1}$.
Moreover, $\ol u_n\geq u$ on $\O\cap(\partial B_{R+n}\cup\partial
B_R)$. Consequently, applying the \MP\ in any connected
component of $\O\cap(B_{R+n}\backslash\ol B_R)$ (where $c<0$) we get
$$\fa n\in\N,\ x\in\O\cap(B_{R+n}\backslash
B_R),\quad u(x)\leq e^{R(\gamma+\ol\sigma)-\ol\sigma|x|}+
e^{(R+n)(\gamma+\ul\sigma)-\ul\sigma|x|}.$$ Letting $n$ go to
infinity in the above inequality yields
$$\fa x\in \O\backslash B_R,\quad
u(x)\leq e^{R(\gamma+\ol\sigma)-\ol\sigma|x|},$$ which concludes
the proof.
\end{proof}

It is not hard to see that the upper bounds for $\gamma$ and $\eta$ are optimal.

\begin{proof}[Proof of Proposition \ref{pro:c<0}]
Proposition \ref{pro:l1''} part (ii) yields
$$\lim_{r\to\infty}\l(-L,\O\backslash\ol{B_r})\geq
\lim_{r\to\infty}
\l''(-L,\O\backslash\ol{B_r})\geq
\lim_{r\to\infty}(-\sup_{\O\backslash\ol{B_r}}c)=
-\xi.$$
Hence, if $\xi<0$ and $\l(-L,\O)>0$ we find that $\l''(-L,\O)>0$ by
\thm{equivalence}. Then the MP holds due to \thm{MP}.
Suppose now that $\l(-L,\O)<-\xi$ (which is the case if $\xi<0$ and
$\l(-L,\O)\leq0$). \thm{mgi} implies that
the eigenfunction $\vp_1$ associated with $\l(-L,\O)$ has minimal growth at
infinity.
Since $v\equiv1$ satisfies $(L+\l(-L,\O))v<0$ a.e.~in $\O\backslash
B_\rho$, for $\rho$ large enough, Definition \ref{def:mgi} implies
that $\vp_1$ is bounded. This concludes the proof of statement (i) and, owing to
Propositions \ref{pro:mgi} and \ref{pro:decay}, statement (ii) also follows.
\end{proof}

\begin{remark}\label{rem:c<0}
The hypothesis $\xi<0$ in Proposition \ref{pro:c<0} part (i) is sharp.
Indeed, we can construct an operator $L$ in $\R$, with a negative zero-order
term vanishing at $\pm\infty$, for which $\l(-L,\R)>0$ but the MP does not
hold. To this aim, consider a
nondecreasing odd function $b\in C^0(\R)$ such that $b=2$ in
$(1/\sqrt3,+\infty)$. Direct computation shows that the function
$u(x):=2-(x^2+1)^{-1}$ satisfies 
$$\fa x\in\R,\quad c(x):=-\frac{u''+b(x)u'}u<0,\qquad
\lim_{x\to\pm\infty}c(x)=0.$$ Defining the operator $L$ by
$Lv:=v''+b(x)v'+c(x)v$, we get $Lu=0$ in $\R$. It is easily
seen that the function $\phi$ defined by
$\phi(x):=e^{-|x|}$ for $|x|\geq1/\sqrt3$ can be extended to the whole line as
a positive smooth
function satisfying $\phi''+b(x)\phi'+\e\phi<0$ in $\R$, for some $\e>0$. As a
consequence, $\l(-L,\R)\geq\e$. Note that if instead of $\R$ we consider the
half line $\R^+$, we still have $\l(-L,\R^+)\geq\e$ and $u-1$ violates
the MP there. 
\end{remark}


\section{Continuous dependence of $\l$ with respect to the coefficients}
\label{sec:coefficients}

We know from statements (vii), (viii) of Proposition \ref{l1} that $\l$ is
Lipschitz-continuous 
(using the $L^\infty$ norm) in its dependence on the
coefficients $b_i$ and $c$. 
Let us show that, if $\O=\R^N$ and the coefficients are H\"older continuous, 
Schauder's estimates and Harnack's inequality imply the Lipschitz-continuity 
with respect to the $a_{ij}$ too. We point out that it is possible to use 
$\sup_{x\in\O}\|\.\nolinebreak\|_{L^p(B_1(x))}$, $p>1$, instead of the 
$L^\infty$ norm and to deal with discontinuous $b_i$, $c$. This was shown by 
A.~Ancona in Theorem 2' of \cite{Ancona3} using much more involved arguments 
than the simple observation presented below.

\begin{proposition}\label{pro:l1Lipsmooth}
Let $L_k=a_{ij}^k(x)\partial_{ij}+b_i(x)\partial_i+c(x)$, $k=1,2$,
be two uniformly elliptic operators with coefficients in $C^{0,\delta}(\R^N)$, 
$\delta\in(0,1)$. Then,
$$|\l(-L_1,\R^N)-\l(-L_2,\R^N)|\leq
C\sum_{i,j=1}^N\|a_{ij}^1-a_{ij}^2\|_{L^\infty(\R^N)},$$ where $C$
depends on $N$, the ellipticity constants of the operators and the
H\"older norms of the coefficients.
\end{proposition}

\begin{proof}
For $k\in\{1,2\}$, let $\vp_k$ be an eigenfunction of $-L_k$ in $\R^N$
associated with $\l(-L_k,\R^N)$, provided by Proposition \ref{l1} part (v).
We know that $\vp_k\in C^{2,\delta}(\R^N)$.
It holds that
$$(L_2+\l(-L_1,\R^N))\vp_1=(a_{ij}^2-a_{ij}^1)\partial_{ij}\vp_1
\quad\text{in }\R^N.$$ By Schauder's interior estimates
(see, e.g., Theorem 6.2 in \cite{GT})
there exists $h>0$, only depending on $N$, the ellipticity constants and the
H\"older norms of the coefficients of $L_1$, such that, for
$x\in\R^N$, $\|\vp_1\|_{C^2(B_1(x))} \leq
h\|\vp_1\|_{L^\infty(B_2(x))}$. Hence, Harnack's inequality yields
$$\fa x\in\R^N,\quad
\|\vp_1\|_{C^2(B_1(x))}\leq C\inf_{B_2(x)}\vp_1 \leq
C\vp_1(x),$$ for some positive constant $C$. As a consequence,
$(L_2+\lambda)\vp_1\leq0$ in $\R^N$, with
$$\lambda=\l(-L_1,\R^N)-C\sum_{i,j=1}^N\|a_{ij}^1-a_{ij}^2\|_{L^\infty(\R^N)}.$$
Taking $\phi=\vp_1$ in the definition of $\l(-L_2,\R^N)$, we then
derive
$$\l(-L_2,\R^N)\geq\l(-L_1,\R^N)-
C\sum_{i,j=1}^N\|a_{ij}^1-a_{ij}^2\|_{L^\infty(\R^N)}.$$
Exchanging the roles of $L_1$ and $L_2$, one gets the two-sided
inequality.
\end{proof}

Next, we derive a semicontinuity property under some weak convergence hypotheses on
the coefficients, as well as a continuity result when $\O=\R^N$ and the
limit operator has continuous coefficients.

\begin{proposition}\label{thm:l1C}
Let $(L_n)_{n\in\N}$ be a sequence of operators in $\O$
of the type
$$L_n u=a_{ij}^n(x)\partial_{ij} u + b_i^n(x)\partial_i u
+c^n(x)u.$$ 
The following properties hold true:
\begin{enumerate}[(i)]
\item if for any $r>0$, the sequences $(a_{ij}^n)_{n\in\N}$,
$(b_i^n)_{n\in\N}$, $(c^n)_{n\in\N}$ are bounded in 
$L^\infty(\O\cap B_r)$, the $(a_{ij}^n)$ are in $C(\ol\O)$ with
smallest eigenvalues $\ul\alpha^n$ satisfying $\inf_{n\in\N}\inf_{\O\cap 
B_r}\ul\alpha^n>0$, and there is $p>1$ such that $a_{ij}^n\to a_{ij}$ in
$L^p_\loc(\O)$ and $b_i^n\rightharpoonup b_i,\
c^n\rightharpoonup c$ in $L^1_\loc(\O)$, then
$$\l(-L,\O)\geq\limsup_{n\to\infty}\l(-L_n,\O);$$


\item if $\O=\R^N$, $L$ is uniformly elliptic, $a_{ij}\in C^{0,\delta}(\R^N)$, 
the $b_i,\ c$ are bounded and uniformly continuous and 
$a_{ij}^n\to a_{ij},\
b_i^n\to b_i,\ c^n\to c$ in $L^\infty(\R^N)$, then
$$\l(-L,\R^N)=\limn\l(-L_n,\R^N).$$
\end{enumerate}
\end{proposition}

\begin{proof}
We write for short $\l:=\l(-L,\O)$ and $\l^n:=\l(-L_n,\O)$. By hypothesis, in 
both cases (i) and (ii), the sequence $(\l^n)_{n\in\N}$ is bounded from 
above due to Proposition \ref{l1} part (ii).

(i) Consider a subsequence of $(\l^n)_{n\in\N}$
(that we still call $(\l^n)_{n\in\N}$) tending to
$\lambda^*:=\limsup_{n\in\N}\l^n$. We know that $\lambda^*<+\infty$. 
Let us suppose that $\lambda^*>-\infty$, because otherwise there 
is nothing to prove.
For  $n\in\N$, let $\vp^n$ be a generalized principal
eigenfunction associated with $\l^n$, normalized by $\vp^n(x_0)=1$,
where $x_0$ is a given point in $\O$. By usual arguments, the
$\vp^n$ converge (up to subsequences) in $C^1_\loc(\O)$ and weakly
in $W^{2,q}_\loc(\O)$, $\fa q<\infty$, to a nonnegative function
$\vp\in W^{2,q}_\loc(\O)$ satisfying $\vp(x_0)=1$. Then, it easily
follows from the hypotheses that $L_n\vp^n$ converges to $L\vp$ in
the sense of $\mc{D}'(\O)$.
%
Therefore, $(L+\lambda^*)\vp=0$ in $\mc{D}'(\O)$ and thus, as
$\vp\in W^{2,q}_\loc(\O)$, also a.e.~in $\O$. The \SMP\ then
yields $\vp>0$ in $\O$. Consequently, taking $\phi=\vp$ in \eq{l1}
we derive $\l\geq\lambda^*$.


(ii) Suppose first that the $b_i$, $c$ are uniformly H\"older
continuous. Arguing as in the proof of Proposition
\ref{pro:l1Lipsmooth} and then using Proposition \ref{l1} parts (vii), (viii),
we can find a positive constant $C$ such that, for $n\in\N$,
$$\l^n\geq \l-C\left(\sum_{i,j=1}^N
\|a_{ij}^n-a_{ij}\|_{L^\infty(\R^N)}+\sum_{i=1}^N
\|b_i^n-b_i\|_{L^\infty(\R^N)}\right)-\|c^n-c\|_{L^\infty(\R^N)}.$$ The
result follows from the above inequality and statement (i).

In order to deal with $b_i,\ c$ uniformly continuous, for any
fixed $\e>0$ consider some smooth functions $b_i^\e,\ c^\e$
satisfying
$$\|b_i-b_i^\e\|_{L^\infty(\R^N)}\leq\e,\qquad
\|c-c^\e\|_{L^\infty(\R^N)}\leq\e,$$ obtained for instance by
convolution with a mollifier (this is where the uniform continuity
of $b_i$, $c$ is required). Then, define the operators
$$L^\e:=a_{ij}(x)\partial_{ij} +b_i^\e(x)\partial_i
+c^\e(x),$$
$$L_n^\e:=a_{ij}^n(x)\partial_{ij} +(b_i^n(x)-b_i(x)+b_i^\e(x))\partial_i
+(c^n(x)-c(x)+c^\e(x)),$$ and call $\l^\e:=\l(-L^\e,\R^N),\
\l^{n,\e}:=\l(-L_n^\e,\R^N)$. Since $L^\e$ has H\"older continuous
coefficients, we know that there exists $n_\e\in\N$ such that, for
$n\geq n_\e$, $|\l^{n,\e}-\l^\e|\leq\e$. Hence, by statements (vii), (viii) of
Proposition \ref{l1} there exists a positive constant $C'$, 
independent of $n$ and $\e$, such that 
$$\fa n\geq n_\e,\quad
|\l^n-\l|\leq|\l^n-\l^{n,\e}|+|\l^{n,\e}-\l^\e|+|\l-\l^\e|\leq
(2C'+1)\e.$$
\end{proof}

In the last part of this section, we investigate the behavior of
$\l$ as the zero and the second order terms blow up as well as
when the ellipticity degenerates.

For $\gamma\in\R$, consider the operator
$$L^c_\gamma u:=a_{ij}(x)\partial_{ij} u+b_i(x)\partial_i u+\gamma c(x)u.$$
We set $\l^c(\gamma):=\l(-L^c_\gamma,\O)$.

\begin{theorem}\label{thm:l1c}
The function $\l^c:\R\to[-\infty,+\infty)$ is concave and satisfies the
following properties:
\begin{enumerate}[(i)]
\item $\l^c(0)\geq0$;

\item if $c$ is lower semicontinuous then
$$\lim_{\gamma\to+\infty}
\frac{\l^c(\gamma)}\gamma=-\sup_\O c;$$

\item if $c$ is
upper semicontinuous then
$$\lim_{\gamma\to-\infty}
\frac{\l^c(\gamma)}\gamma=-\inf_\O c.$$
\end{enumerate}
Moreover, if $c$ is bounded then $\l^c$ is uniformly Lipschitz-continuous
with Lipschitz constant $\|c\|_{L^\infty(\O)}$.
\end{theorem}

\begin{proof}
The concavity and the Lipschitz-continuity follow from
Proposition \ref{l1} part (vii). Statement (i) is an immediate consequence of
definition \eq{l1}. Let us prove (ii).
Proposition \ref{l1} part (ii) implies that, for $\gamma>0$,
$\l^c(\gamma)\geq-\gamma\sup c$. Hence, to prove the statement it is
sufficient to show that $\limsup_{\gamma\to+\infty}\l^c(\gamma)/\gamma
\leq-\sup_\O c$.
The lower semicontinuity
of $c$ implies that, for any given $\e>0$, there exists a ball
$B\subset\O$ such that $c>\sup_\O c-\e$ in $B$.
Let $\lambda_B$ and $\vp$ denote the Dirichlet \pe\ and eigenfunction of
the operator $-a_{ij}(x)\partial_{ij}-b_i(x)\partial_i$ in $B$.
For $\gamma>0$, the function $\vp$ satisfies, a.e.~in $B$,
$$(L^c_\gamma+\gamma(-\sup_\O c+2\e))\vp=
[-\lambda_B+\gamma(c(x)-\sup_\O c+2\e)]\vp>
(\e\gamma-\lambda_B)\vp.$$
Therefore, for $\gamma\geq\lambda_B/\e$,
taking $\phi=\vp$ in \eq{l1'} we get $\l'(-L^c_\gamma,B)\leq\gamma(-\sup_\O c+2\e)$.
Since $\l'(-L^c_\gamma,B)=\l(-L^c_\gamma,B)$, Proposition \ref{l1} part (iii) yields
$$-\sup_\O c+2\e\geq
\limsup_{\gamma\to+\infty}\frac{\l(-L^c_\gamma,B)}\gamma\geq
\limsup_{\gamma\to+\infty}\frac{\l^c(\gamma)}\gamma.$$
The proof of (ii) is thereby achieved due to arbitrariness of $\e$.
Statement (iii) follows from (ii) by replacing the operator $L$ with
$a_{ij}(x)\partial_{ij}+b_i(x)\partial_i-c(x)$.
\end{proof}

\begin{remark}
In the proof of \thm{l1c}, we have shown that
$$\lim_{\gamma\to+\infty}\frac{\l^c(\gamma)}\gamma\leq
-\sup\{k\in\R\ :\ \ex \text{a ball }B\subset\O\text{ such that }
c(x)\geq k\text{ in }B\}.$$ Clearly, if $c$ is lower
semicontinuous then the right-hand side of the above inequality
coincides with $-\sup_\O c$.
\end{remark}

For $\alpha>0$, we define
$$L^a_\alpha u:=\alpha a_{ij}(x)\partial_{ij} u+b_i(x)\partial_i u+c(x)u.$$
We set for brief $\l^a(\alpha):=\l(-L^a_\alpha,\O)$.

\begin{theorem}\label{thm:l1a}
The function $\l^a:\R_+\to[-\infty,+\infty)$ satisfies the following properties:
\begin{enumerate}[(i)]
\item if $L$ has bounded coefficients then
$\l^a$ is locally Lipschitz-continuous on $\R_+$;

\item if $\O$ contains balls of arbitrarily large radius and $L$ is
uniformly elliptic with bounded coefficients, then
$$\liminf_{\alpha\to+\infty}
\l^a(\alpha)\geq-\limsup_{\su{x\in\O}{|x|\to\infty}}
c(x),\qquad\limsup_{\alpha\to+\infty}
\l^a(\alpha)\leq-\liminf_{\su{x\in\O}{|x|\to\infty}} c(x);$$

\item if the $L^a_\alpha$ are self-adjoint then
$\l^a$ is concave and nondecreasing. If in addition $c$ is lower
semicontinuous then
$$\lim_{\alpha\to0^+}\l^a(\alpha)=-\sup_{\O} c.$$
\end{enumerate}
\end{theorem}

\begin{proof}
(i) For $\alpha>0$, we can write $L^a_\alpha=\alpha
L_{1/\alpha}^{b,c},$ with $$L_{1/\alpha}^{b,c}:=
a_{ij}(x)\partial_{ij}+\frac1\alpha b_i(x)\partial_i+ \frac1\alpha
c(x).$$ Therefore,
$\l^a(\alpha)=\alpha\l(-L_{1/\alpha}^{b,c},\O)$. The statement
then follows from statements (vii), (viii) of Proposition \ref{l1}.

(ii)
We make use of the estimate (4.3) in \cite{BHRossi}. It implies that
$$\l^a(\alpha)\leq-\liminf_{\su{x\in\O}{|x|\to\infty}}
\left(c(x)-\frac{|b(x)|^2}{4\alpha\inf\ul\alpha}\right).$$
Consequently,
$$\limsup_{\alpha\to+\infty}\l^a(\alpha)\leq
-\liminf_{\su{x\in\O}{|x|\to\infty}}c(x).$$ In order to
prove that \Fi{liminfalpha}
\liminf_{\alpha\to+\infty}\l^a(\alpha)
\geq-\limsup_{\su{x\in\O}{|x|\to\infty}}c(x), \Ff we define the function
$\phi(x):=\vartheta(\alpha^{-1/8}|x|)$, with
$$\vartheta(\rho):=(e^\rho+e^{-\rho})^{-\alpha^{-1/2}}.$$
As $\vartheta'(\rho)\leq0$ for $\rho\geq0$, it follows that, for
a.e.~$x\in\O$,
\begin{equation*}\begin{split}
a_{ij}\partial_{ij}\phi(x) &=
A(x)\alpha^{-1/4}\vartheta''(\alpha^{-1/8}|x|)
+\alpha^{-1/8}\vartheta'(\alpha^{-1/8}|x|)\frac
{\sum_{i=1}^N a_{ii}(x)-A(x)}{|x|}\\
&\leq A(x)\alpha^{-1/4}\vartheta''(\alpha^{-1/8}|x|),
\end{split}\end{equation*}
where $A(x)=\frac{a_{ij}(x)x_i x_j}{|x|^2}\geq\ul\alpha(x)$.
Thus, direct computation yields
$$L^a_\alpha\phi\leq\left[A(x)\alpha^{1/4}\left((1+\alpha^{-1/2})
g(\alpha^{-1/8}|x|)-1\right)+\norma{b}
\alpha^{-5/8}+c(x)\right]\phi,$$ with
$$g(\rho):=\left(\frac{e^\rho-e^{-\rho}}{e^\rho+e^{-\rho}}\right)^2.$$
For given $\e>0$, let $R>0$ be such that
$c\leq\limsup_{|x|\to\infty}c(x)+\e$ a.e.~in $\O\backslash B_R$.
For $\alpha$ large enough and for a.e.~$x\in\O\cap  B_R$ it holds true that
$g(\alpha^{-1/8}|x|)\leq 1/2$, and then that
$$L^a_\alpha\phi\leq\left(\frac12A(x)\alpha^{1/4}
(-1+\alpha^{-1/2})+\norma{b} \alpha^{-5/8}+c(x)\right)\phi.$$ On
the other hand, for a.e.~$x\in\O\backslash B_R$ we find
$$L^a_\alpha\phi\leq\left(A(x)\alpha^{-1/4}+\norma{b}
\alpha^{-5/8}+\limsup_{|x|\to\infty}c(x)+\e\right)\phi.$$
Consequently, $L^a_\alpha\phi\leq(\limsup_{|x|\to\infty}c(x)+2\e)
\phi$ a.e.~in $\R^N$ for $\alpha$ large enough. Therefore, by
definition \eq{l1} we obtain
$$\liminf_{\alpha\to+\infty}\l^a(\alpha)
\geq-\limsup_{|x|\to\infty}c(x)-2\e,$$ which concludes the proof
due to the arbitrariness of $\e$.

(iii) Proposition \ref{l1} part (vi) implies that the function
$\l^a$ is concave and nondecreasing.
Since $L^a_\alpha=\alpha L^c_{1/\alpha}$, it holds that
$\l^a(\alpha)=\alpha\l^c(1/\alpha)$. The last
statement then follows by applying \thm{l1c} part (ii).
\end{proof}

\appendix

\theoremstyle{definition}
\newtheorem*{Aproperties}{Properties A}

\section{Known results in bounded non-smooth domains}\label{sec:BNV}

Even though in the present paper we are only interested in the
case $\O$ smooth, in some of the proofs we deal with intersections
of smooth domains, which are no longer smooth. This is why we
require some of the tools developed in \cite{BNV} to treat the
non-smooth case. When $\O$ is non-smooth, the Dirichlet boundary
condition has to be relaxed to a weaker sense:
$$u\uguale0\quad\text{(resp.~$u\minore0$)\quad on }\partial\O,$$
which means that, if there is a sequence $\seq{x}$ in $\O$
converging to a point of $\partial\O$ such that $\limn
u_0(x_n)=0$, then
$$\limn u(x_n)=0\quad\text{(resp.~}
\limsup_{n\to\infty} u(x_n)\leq0\text{)},$$ where $u_0$ is the
``boundary function'' associated with the problem (see \cite{BNV}).
We do not need to define the function $u_0$ here since, in the
proofs, we only use the information $u\uguale0$ on smooth portions
of $\partial\O$. It suffices to know that, there, it coincides
with the standard Dirichlet condition. Indeed, it turns out that
if $u\uguale0$ on $\partial\O$ then it can be extended as a
continuous function to every $\xi\in\partial\O$ admitting a so
called ``strong barrier'' by setting $u(\xi)=0$. Since any point
$\xi\in\partial\O$ satisfying the exterior cone condition admits a
strong barrier, it follows that $u$ vanishes continuously on
smooth boundary portions of $\partial\O$.

We now assume that $L$ is uniformly elliptic and that
$$a_{ij}\in C^0(\ol\O),\qquad b_i,c\in L^\infty(\O).$$

\begin{definition}\label{def:refMP}
We say that the operator $L$ satisfies the {\em refined MP\,}\
in $\O$ if every function $u\in W^{2,N}_\loc(\O)$ such that
$$Lu\geq0 \ \text{ a.e.~in }\O,\qquad
\sup_{\O}u<\infty,\qquad
u\minore0\ \text{ on }\partial\O,$$
satisfies $u\leq0$ in $\O$.
\end{definition}

\begin{Aproperties}[\cite{BNV}]\label{A}
Let $\O$ be a general bounded domain. Then, the following properties
hold:

\begin{itemize}

\item[A.1\ ] There exists a positive bounded function
$\vp_1\in W^{2,p}_{loc}(\O)$, $\fa p<\infty$, called
generalized principal eigenfunction of $-L$ in $\O$, satisfying
$$\left\{\begin{array}{ll}
-L\vp_1=\l(-L,\O)\vp_1 & \text{a.e.~in }\O\\
\vp_1\uguale0 & \text{on }\partial\O;
\end{array}\right.$$
moreover, if $\O$ has a $C^{1,1}$ boundary portion $T\subset\partial\O$,
then $\vp_1\in W^{2,p}_{loc}(\O\cup T)$ and $\vp_1=0$ on $T$;

\item[A.2\ ]
If $\l(-L,\O)>0$ then $L$ satisfies the refined MP in $\O$;

\item[A.3\ ] If $\phi\in W^{2,N}_{loc}(\O)$ is bounded from above and satisfies
$$-L\phi\leq\l(-L,\O)\phi\quad\text{a.e.~in }\O,
\qquad \phi\minore0\quad\text{on }\partial\O,$$
then $\phi$ is a constant multiple of the generalized principal eigenfunction $\vp_1$;

\item[A.4\ ] If there exists a positive function
$\phi\in W^{2,N}_{loc}(\O)$ satisfying
$$L\phi\leq0\quad\text{a.e.~in }\O,$$
then either $\l(-L,\O)>0$ or $\l(-L,\O)=0$ and
$\phi$ is a constant multiple of $\vp_1$;

\item[A.5\ ] If $\l(-L,\O)>0$ then, given $f\in L^N(\O)$,
there is a unique bounded solution
$u\in W^{2,N}_{loc}(\O)$ satisfying
$$\left\{\begin{array}{ll}
Lu=f & \text{a.e.~in }\O\\
u\uguale0 & \text{on }\partial\O;
\end{array}\right.$$
moreover, if $\O$ has a $C^{1,1}$ boundary portion $T\subset\partial\O$,
then $u\in W^{2,N}_{loc}(\O\cup T)$ and $u=0$ on $T$;

\item[A.6\ ] If $\l(-L,\O)>0$ and $u\in W^{2,N}_{loc}(\O)$ is bounded
above and satisfies
$$Lu\geq f\quad\text{a.e.~in }\O,\qquad
u\minore\beta\quad\text{on }\partial\O,$$
for some nonpositive function $f\in L^N(\O)$ and nonnegative constant $\beta$,
then
$$\sup_\O u\leq \beta+A\left(\|f\|_{L^N(\O)}+\beta\sup c^+|\O|^{1/N}\right),$$
where $A$ only depends on $\O$, $\l(-L,\O)$, $\inf\ul\alpha$
and the $L^\infty$ norms of $a_{ij}$, $b_i$, $c$.

\end{itemize}
\end{Aproperties}

Property A.1 is Theorem 2.1 in \cite{BNV}, except for the
improved regularity of $\vp_1$ near the smooth boundary portion
$T$.
The latter follows from the standard local boundary estimate, even though a
technical difficulty arises because $\vp_1$ does not belong to $W^{2,p}(\O)$.
However, it can be overcome using the same approximation argument as in the
proof of
Lemma 6.18 in \cite{GT}.
%
%
The same is true for
the last statement of A.5. The other properties refer to the
following results of \cite{BNV}: A.2 is Theorem 1.1,
A.3 is Corollary 2.2, A.4 is Corollary 2.1, A.5
is Theorem 1.2, A.6 is Theorem 1.3.


\section{The inhomogeneous boundary Harnack inequality}\label{sec:bHarnack}

Using the refined Alexandrov-Bakelman-Pucci estimate, we extend the boundary
Harnack inequality - \thm{bHarnack} - to solutions of inhomogeneous Dirichlet
problems.

\begin{proposition}\label{pro:bHarnack}
Let $\O$ be a bounded domain and $\O'$ be an open subset of $\O$
such that $T:=\partial\O\cap (\O'+B_\eta)$ is of class $C^{1,1}$,
for some $\eta>0$.
Then, any nonnegative function $u\in W^{2,N}_{loc}(\O\cup T)$ such that
$$L^N(\O)\ni Lu\leq0 \quad \text{a.e.~in }\O,$$
satisfies
$$\sup_{\O'} u\leq \sup_T u+C\left(\inf_{\O^\delta}u+\|Lu\|_{L^N(\O)}+\sup_\O
c^+\,\sup_T u\right),$$
for all $\delta>0$ such that $\O^\delta\neq\emptyset$, with $C$
depending on
$N$, $\O$, $\delta$, $\eta$, $\inf\ul\alpha$, the
$L^\infty$ norms of $a_{ij}$, $b_i$, $c$ and $\l(-L,\O)$.
\end{proposition}

\begin{proof}
Suppose first that $\l(-L,\O)\leq0$. If $u$ vanishes somewhere in
$\O$ then $u\equiv0$ by the \SMP\ and the statement trivially holds.
If $u$ is positive then $\l(-L,\O)=0$. Thus, by Property A.4, 
$u$ is the generalized
principal eigenfunction of $-L$ in $\O$. In
particular, $Lu=0$ and $u=0$ (in the classical sense) on $T$.
The result then follows from \thm{bHarnack}.
Consider now the case $\l(-L,\O)>0$.
Set $f:=Lu$
and let $\chi:\R^N\to[0,1]$ be a smooth function such
that
$$\chi=1\quad\text{in }\O'+B_{\eta/4},\qquad
\chi=0\quad\text{outside }\O'+B_{\eta/2}.$$
Let $v\in W^{2,N}_{loc}(\O\cup T)\cap L^\infty(\O)$ be the solution
of the problem
$$\left\{\begin{array}{ll}
Lv=f & \text{a.e.~in }\O\\
v\uguale\chi u & \text{on }\partial\O.\\
\end{array}\right.$$
It is given by $v=w+\chi u$, where $w$ is the unique bounded solution of
$$\left\{\begin{array}{ll}
Lw=f-L(\chi u) & \text{a.e.~in }\O\\
w\uguale0 & \text{on }\partial\O,\\
\end{array}\right.$$
provided by Property A.5 (note that $\chi u\in W^{2,N}(\O)$). 
We have $0\leq v\leq u$ by the
refined MP - which holds due to Property A.2. The refined
Alexandrov-Bakelman-Pucci estimate - Property A.6 - yields
$$\sup_{\O} v\leq\sup_T u+
A\left(\|Lu\|_{L^N(\O)}+\sup_\O c^+\,
\sup_T u\right),$$
where $A$ depends on $\O$, $\l(-L,\O)$ and the coefficients of $L$.
Applying \thm{bHarnack} to $u-v$, we obtain
$$\sup_{\O'}(u-v)\leq C'\inf_{\O^\delta}(u-v)\leq C'\inf_{\O^\delta}u.$$
The result then follows by gathering the above inequalities.
\end{proof}



\section*{Acknowledgments}

\thanks{

The research leading to these results has received funding from the European
Research Council under the European Union's Seventh Framework Programme
(FP/2007-2013) / ERC Grant Agreement n.321186 - ReaDi -Reaction-Diffusion
Equations, Propagation and Modelling. Part of this work was done while Henri
Berestycki was visiting the University of Chicago. He was also supported by an
NSF FRG grant DMS - 1065979. Luca Rossi was partially supported by the
Fondazione CaRiPaRo Project ``Nonlinear Partial Differential Equations:
models, analysis, and control-theoretic problems''.
}


\frenchspacing

\def\cprime{$'$} \def\polhk#1{\setbox0=\hbox{#1}{\ooalign{\hidewidth
  \lower1.5ex\hbox{`}\hidewidth\crcr\unhbox0}}}
  \def\cfac#1{\ifmmode\setbox7\hbox{$\accent"5E#1$}\else
  \setbox7\hbox{\accent"5E#1}\penalty 10000\relax\fi\raise 1\ht7
  \hbox{\lower1.15ex\hbox to 1\wd7{\hss\accent"13\hss}}\penalty 10000
  \hskip-1\wd7\penalty 10000\box7}

\addcontentsline{toc}{section}{References}
\bibliographystyle{plain}

\end{document}